\documentclass[preprint,3p]{elsarticle}

\usepackage{amssymb}

\usepackage{amsmath,amsthm}   
\usepackage{lipsum}
\usepackage{amsfonts}
\usepackage{graphicx}
\usepackage{epstopdf}
\usepackage{algorithmic}
\ifpdf
  \DeclareGraphicsExtensions{.eps,.pdf,.png,.jpg}
\else
  \DeclareGraphicsExtensions{.eps}
\fi
\usepackage{tikz}
\usetikzlibrary{arrows,decorations,patterns,positioning,automata,shadows,fit,shapes,calc,decorations.markings,backgrounds,scopes,decorations.text}
\tikzset{>=latex}
\usepackage{graphicx}     
\usepackage{quiver}

\newtheorem{theorem}{Theorem}
\newtheorem{lemma}{Lemma}

\newcommand{\nullspace}[1]{\hbox{\textbf{N}}(#1)}
\newcommand{\colspace}[1]{\hbox{\textbf{C}}(#1)}

\myfooter[C]{} 

\begin{document}

\begin{frontmatter}



\title{Generalized Inverses of Matrix Products:\\From Fundamental Subspaces to Randomized Decompositions}


\author[a]{Micha\l\ P. Karpowicz}
\ead{m.karpowicz@samsung.com}\affiliation[a]{organization={Samsung AI Center},city={Warsaw},country={Poland}}
\author[b]{Gilbert Strang}
\ead{gilstrang@gmail.com}
\affiliation[b]{organization={MIT},state={Massachusetts},country={USA}}

\begin{abstract}

We investigate the Moore-Penrose pseudoinverse and generalized inverse of a matrix product $A=CR$ to establish a unifying framework for generalized and randomized matrix inverses. This analysis is rooted in first principles, focusing on the geometry of the four fundamental subspaces. We examine:

\begin{enumerate}
\item The reverse order law, $A^+ = R^+C^+$, which holds when $C$ has independent columns and $R$ has independent rows.
\item The universally correct formula, $A^+ = (C^+CR)^+(CRR^+)^+$, providing a geometric interpretation of the mappings between the involved subspaces.
\item A new generalized randomized formula, $A^+_p = (P^TA)^+P^TAQ(AQ)^+$, which gives $A^+_p = A^+$ if and only if the sketching matrices $P$ and $Q$ preserve the rank of $A$, i.e.,  $\mathrm{rank}(P^TA) = \mathrm{rank}(AQ) = \mathrm{rank}(A)$.
\end{enumerate}

The framework is extended to generalized $\{1,2\}$-inverses and specialized forms, revealing the underlying structure of established randomized linear algebra algorithms, including randomized SVD, the Nyström approximation, and CUR decomposition. We demonstrate applications in sparse sensor placement and effective resistance estimation. For the latter, we provide a rigorous quantitative analysis of an approximation scheme, establishing that it always underestimates the true resistance and deriving a worst-case spectral bound on the error of resistance differences.

\end{abstract}


\end{frontmatter}



\section{Introduction}

The reverse-order law for pseudoinverse is not generally true, as the following example shows:
\begin{gather*}
C = \left[\begin{array}{rrr}1&0\\\end{array}\right],\  
R = \left[\begin{array}{rrr}1\\1\\\end{array}\right],\  
(CR)^+ = \left[\begin{array}{r}1\\\end{array} \right],\ 
R^+C^+ = \left[\begin{array}{rrr}\dfrac{1}{2}&\dfrac{1}{2}\\\end{array}\right]
\left[\begin{array}{rrr}1\\0\\\end{array}\right]
=\left[\dfrac{1}{2}\right].
\end{gather*}
Here, $C$ has a full row rank and $R$ has a full column rank. We need to have it the other way around as a sufficient condition in Theorem~\ref{thm:rol-formula}. Then Theorem~\ref{thm:correct-formula} will give a formula for the pseudoinverse of $A = CR$ that applies in every case. Theorem~\ref{thm:general-MP-formula} will give a general recipe of the pseudoinverse and its randomized form, reconstructing it only in special cases. Then, Theorem~\ref{thm:Nystroem-MP-formula} links the three formulas with the generalized Nystroem decomposition.

This paper is partially pedagogical and partially exploratory. First, we take a closer look at the facts and results that are mostly well-known and can be found in the literature on generalized inverses, including \cite{ben2003generalized,campbell2009generalized,ilic2017algebraic}. But, in doing so, we contemplate the first principles behind mappings connecting the four fundamental subspaces. In the alternate proofs we present, we purposefully avoid Penrose identities and SVD-based arguments and focus on the underpinnings of generalized inverses. We believe that may be inspiring and illuminating enough to provide an even better understanding of the structure and internal mechanics of pseudoinverse. 

That approach leads us next to the exploratory and research-oriented part of the paper, where we derive the general formula for randomized pseudoinverse and use it to derive some previously unknown results in subsection \ref{subsection:Effective resistance calculations}. The generalized formula produces pseudoinverse only under special circumstances with rank-preserving sampling matrices. Otherwise, it is a recipe for low-rank approximations and other known or possibly unknown formulas. We believe that is a valuable contribution to the well-established field of randomized linear algebra, which is beautifully presented in \cite{halko2011finding,murray2023randomized,martinsson2020randomized}.
Although SVD and its randomized version are regarded as leading algorithms to compute the Moore-Penrose inverse, whenever that is truly necessary, we show how to introduce parameters into the pseudoinverse formula. These can be optimized or used to study the properties of generalized inverse mapping. 

\section{The reverse order law for pseudoinverse}

When the $m$ by $r$ matrix $C$ has $r$ independent columns (full column rank~$r$), and the $r$ by $n$ matrix $R$ has $r$ independent rows (full row rank $r$), the pseudoinverse $C^+$ is the left inverse of $C$, and the pseudoinverse $R^+$ is the right inverse of $R$\,:
$$
C^+=(C^\mathrm{T} C)^{-1}C^\mathrm{T}\text{ \;has\; } C^+C=I_{\text{\small$r$}}\ \mathrm{and}\  R^+=R^\mathrm{T}(RR^\mathrm{T})^{-1}\text{ \;has\; } RR^+=I_{\text{\small$r$}}.
$$
In this case, the $m$ by $n$ matrix $A=CR$ will also have rank $r$ and the statement correctly claims that its $n$ by $m$ pseudoinverse is $A^+=R^+C^+$.

\vspace{10pt}
\begin{theorem}\label{thm:rol-formula}
The pseudoinverse $A^+$ of a product $A=CR$ is the product of the pseudoinverses $R^+C^+$ if $C \in \mathbb{R}^{m \times r}$ has full column rank $r$ and $R \in \mathbb{R}^{r \times n}$ has full row rank $r$.
\end{theorem}
\vspace{10pt}

The simplest proof verifies the four Penrose identities \cite{penrose1955generalized} that determine the pseudoinverse $A^+$ of any matrix $A$\,:

\vspace{-11pt}
\begin{equation}
\!\!
AA^+A=A,
\qquad\quad 
A^+AA^+=A^+,
\qquad 
(A^+A)^\mathrm{T}=A^+A,
\qquad 
(AA^+)^\mathrm{T}=AA^+.
\!\!
\label{eq:Penrose}
\end{equation}

\noindent
Our goal is a different proof of $A^+=R^+C^+$, starting from first principles. We begin with the \textbf{four fundamental subspaces} associated with any $m$ by $n$ matrix $A$ of rank $r$. Those are the column space $\hbox{\textbf{C}}$ and nullspace $\hbox{\textbf{N}}$ of $A$ and $A^\mathrm{T}$.

Note that every matrix $A$ gives an invertible map (Figure \ref{fig:1}) from its row space $\hbox{\textbf{C}}(A^\mathrm{T})$ to its column space $\hbox{\textbf{C}}(A)$. If $\boldsymbol{x}$ and $\boldsymbol{y}$ are in the row space and $A\boldsymbol{x}=A\boldsymbol{y}$, then $A(\boldsymbol{x}-\boldsymbol{y})=\boldsymbol{0}$. Therefore $\boldsymbol{x}-\boldsymbol{y}$ is in the nullspace $\nullspace{A}$ as well as the row space. So $\boldsymbol{x}-\boldsymbol{y}$ is orthogonal to itself and $\boldsymbol{x}=\boldsymbol{y}$.

The pseudoinverse $A^+$ in Figure \ref{fig:2} inverts the row space to column space map in Figure \ref{fig:1} when $\hbox{\textbf{N}}(A^T) = \hbox{\textbf{N}}(A^+)$. If $\boldsymbol{b}_c$ and $\boldsymbol{d}_c$ are in the column space $\hbox{\textbf{C}}(A)$ and $\boldsymbol{x} = A^+\boldsymbol{b}_c=A^+\boldsymbol{d}_c$, then $A^+(\boldsymbol{b}_c-\boldsymbol{d}_c)=\boldsymbol{0}$ and $\boldsymbol{b}_c-\boldsymbol{d}_c$ is in the nullspace $\hbox{\textbf{N}}(A^+)$. Therefore, if $\hbox{\textbf{N}}(A^+) = \hbox{\textbf{N}}(A^T)$, then $\boldsymbol{b}_c-\boldsymbol{d}_c$ is in the nullspace of $A^T$ and the column space of $A$, it is orthogonal to itself and $(\boldsymbol{b}_c-\boldsymbol{d}_c)^T(\boldsymbol{b}_c-\boldsymbol{d}_c) = 0$ with $\boldsymbol{b}_c=\boldsymbol{d}_c$. For all vectors $\boldsymbol{b}_n$ in the orthogonal complement of $\hbox{\textbf{C}}(A)$, we have $A^\mathrm{T}\boldsymbol{b}_n=\boldsymbol{0}$ and  $A^+\boldsymbol{b}_n=\boldsymbol{0}$. Then Theorem \ref{thm:rol-formula} shows $A^+$ is the inverse map.

That upper map from the row space $\hbox{\textbf{C}}(A^\mathrm{T})$ to the column space $\hbox{\textbf{C}}(A)$ is inverted by the pseudoinverse $A^+$ in Figure \ref{fig:2}. And the nullspace of $\boldsymbol{A^+}$ is the same as the nullspace of $\boldsymbol{A^\mathrm{T}}$---the orthogonal complement of $\hbox{\textbf{C}}(A)$. \textbf{Thus $\boldsymbol{1/0=0}$ for $\boldsymbol{A^+}$}, the pseudoinverse maps the orthogonal complement of the column space to the zero vector. In the extreme case, the pseudoinverse of $A=$ zero matrix $(m$ by $n$) is $A^+=$ zero matrix ($n$ by $m$).

\begin{figure}[!t]
\centering
\scalebox{0.85}{
\begin{tikzpicture}[scale=2.3]
\node at (0.8,2.6) {\scalebox{1.1}{$\hbox{\textbf{C}}(\boldsymbol{A^\mathrm{T}})$}};
\draw (1.07,3.40)--(1.88,2.59)--(0.35,1.03)--(1.35,0.03)--(1.87,0.56)--(0.06,2.37)--cycle;
\draw (0.79,1.65)--(0.69,1.55)--(0.77,1.46);
\node at (0.30,3.12) {dim $\boldsymbol{r}$};
\node at (0.37,0.47) {dim $\boldsymbol{n-r}$};
\node at (1.0,0.8) {\scalebox{1.1}{$\hbox{\textbf{N}}(\boldsymbol{A})$}};
\node at (0.49,1.55) {\scalebox{1.1}{$\hbox{\textbf{R}}^{\text{\scriptsize{$\boldsymbol{n}$}}}$}};
\node at (1.4,2.35) {$\boldsymbol{x}_\text{\footnotesize ${\boldsymbol{r}}$}$};
\draw [fill=black](1.6, 2.3) circle (0.8pt);
\draw [fill=black](4.25, 2.3) circle (0.8pt);
\draw [fill=black](1.90,2.05) circle (0.8pt);
\draw [fill=black](1.15,1.28) circle (0.8pt);
\draw [dashed](1.6, 2.3)--(1.90,2.05)--(1.15,1.28);
\draw[decoration={markings, mark=at position 0.6 with {\arrow[scale=2]{>}}},postaction={decorate}](1.6, 2.3)--(4.25, 2.3);
\node at (3.15,2.52) {$A\boldsymbol{x}_\text{\footnotesize ${\boldsymbol{r}}$}=\boldsymbol{b}_\text{\footnotesize ${\boldsymbol{c}}$}$};
\draw[decoration={markings, mark=at position 0.7 with {\arrow[scale=2]{>}}},postaction={decorate}](1.90,2.05)--(4.25, 2.3);
\node at (2.37,1.9) {$\boldsymbol{x}=\boldsymbol{x}_\text{\footnotesize ${\boldsymbol{r}}$}+\boldsymbol{x}_\text{\footnotesize ${\boldsymbol{n}}$}$};
\node at (3.45,2.05) {$A\boldsymbol{x}=\boldsymbol{b}_\text{\footnotesize ${\boldsymbol{c}}$}$};
\draw[decoration={markings, mark=at position 0.5 with {\arrow[scale=2]{>}}},postaction={decorate}](1.15,1.28)--(4.4, 1.58);
\node at (0.98,1.25) {$\boldsymbol{x}_\text{\footnotesize ${\boldsymbol{n}}$}$};
\node at (2.65,1.25) {$A\boldsymbol{x}_\text{\footnotesize ${\boldsymbol{n}}$}=\boldsymbol{0}$};
\draw (3.79,2.88)--(4.83,3.36)--(5.43,2.06)--(3.11,0.97)--(3.36,0.43)--(4.64,1.05)--cycle;
\draw (4.51,1.63)--(4.57,1.52)--(4.45,1.46);
\node at (5.33,3.12) {dim $\boldsymbol{r}$};
\node at (4.70,2.6) {\scalebox{1.1}{$\hbox{\textbf{C}}(\boldsymbol{A})$}};
\node at (4.59,0.72) {dim $\boldsymbol{m-r}$};
\node at (3.9,1.0) {\scalebox{1.1}{$\hbox{\textbf{N}}(\boldsymbol{A^\mathrm{T}})$}};
\node at (4.79,1.53) {\scalebox{1.1}{$\hbox{\textbf{R}}^{\text{\scriptsize{$\boldsymbol{m}$}}}$}};
\node at (4.25,2.45) {$\boldsymbol{b}_\text{\footnotesize ${\boldsymbol{c}}$}$};%
\end{tikzpicture}
}
\vspace{-5pt}
\caption{
$A\boldsymbol{x}_\text{\footnotesize ${\boldsymbol{r}}$}=\boldsymbol{b}$ is in the column space of $A$ and $A\boldsymbol{x}_\text{\footnotesize ${\boldsymbol{n}}$}=\boldsymbol{0}$. The complete solution to $A\boldsymbol{x}=\boldsymbol{b}$ is $\boldsymbol{x}=$ \textbf{\,one\;} $\boldsymbol{x}_\text{\footnotesize ${\boldsymbol{r}}$}\,+$ \textbf{\,any\;} $\boldsymbol{x}_\text{\footnotesize ${\boldsymbol{n}}$}$. 
}\label{fig:1}
\end{figure}
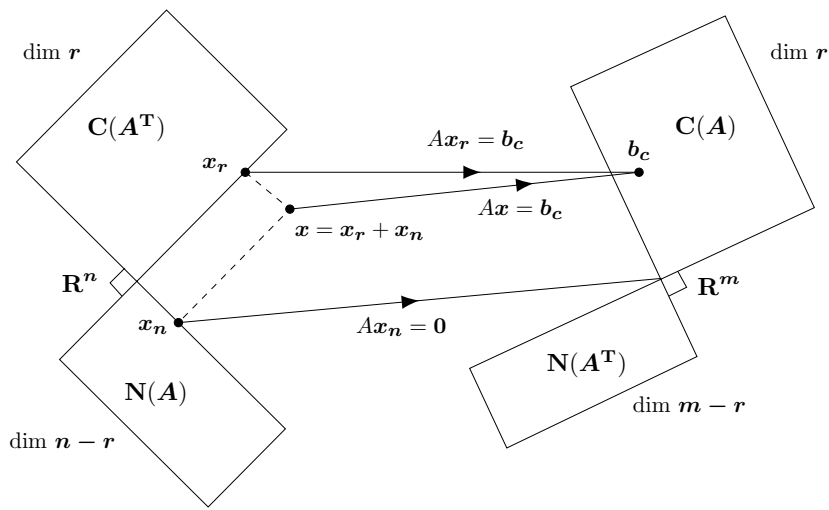

\begin{figure}[ht!]
\centering
\scalebox{0.85}{\hspace{9pt}
\begin{tikzpicture}[scale=2.3]
\draw (1.07,3.40)--(1.88,2.59)--(0.35,1.03)--(1.35,0.03)--(1.87,0.56)--(0.06,2.38)--cycle;
\draw (0.79,1.65)--(0.69,1.55)--(0.77,1.46);
\node at (0.30,3.12) {dim $\boldsymbol{r}$};
\node at (1.05,2.6) {\scalebox{1.1}{$\hbox{\textbf{C}}(\boldsymbol{A^\mathrm{T}})=\hbox{\textbf{C}}(\boldsymbol{A^+})$}};
\node at (0.62,0.27) {dim $\boldsymbol{n-r}$};
\node at (0.9,1.0) {\scalebox{1.1}{$\hbox{\textbf{N}}(\boldsymbol{A})$}};
\node at (0.49,1.55) {\scalebox{1.1}{$\hbox{\textbf{R}}^{\text{\scriptsize{$\boldsymbol{n}$}}}$}};
\node at (1.4,2.35) {$\boldsymbol{x}_\text{\footnotesize ${\boldsymbol{r}}$}$};
\draw [fill=black](1.6, 2.3) circle (0.8pt);
\draw [fill=black](4.6, 2.25) circle (0.8pt);
\draw [fill=black](3.73,1.85) circle (0.8pt);
\draw [fill=black](3.95,1.36) circle (0.8pt);
\draw [dashed](4.6, 2.25)--(3.73,1.85)--(3.95,1.36);
\draw [dashed](1.6, 2.3)--(3.73,1.85);
\draw[decoration={markings, mark=at position 0.3 with {\arrow[scale=2]{<}}},postaction={decorate}](1.6, 2.3)--(4.6, 2.25);
\draw[decoration={markings, mark=at position 0.3 with {\arrow[scale=2]{<}}},postaction={decorate}](1.6, 2.3)--(3.73,1.85);
\node at (3.15,2.42) {$A^{\pmb{+}}\boldsymbol{b}_\text{\footnotesize ${\boldsymbol{c}}$}=\boldsymbol{x}_\text{\footnotesize ${\boldsymbol{r}}$}$};
\node at (3.30,2.05)[rotate=-10] {$\boldsymbol{b}=\boldsymbol{b}_\text{\footnotesize ${\boldsymbol{c}}$}+\boldsymbol{b}_\text{\footnotesize ${\boldsymbol{n}}$}$};
\node at (2.67,1.9) {$A^{\pmb{+}}\boldsymbol{b}=\boldsymbol{x}_\text{\footnotesize ${\boldsymbol{r}}$}$};
\draw[decoration={markings, mark=at position 0.3 with {\arrow[scale=2]{<}}},postaction={decorate}](0.88,1.56)--(3.95,1.36);
\node at (4.02,1.15) {$\boldsymbol{b}_\text{\footnotesize ${\boldsymbol{n}}$}$};
\node at (2.05,1.35) {$A^{\pmb{+}}\boldsymbol{b}_\text{\footnotesize ${\boldsymbol{n}}$}=\boldsymbol{0}$};
\draw (3.79,2.88)--(4.83,3.36)--(5.43,2.06)--(3.11,0.97)--(3.36,0.43)--(4.64,1.05)--cycle;
\draw (4.51,1.63)--(4.57,1.52)--(4.45,1.46);
\node at (5.33,3.12) {dim $\boldsymbol{r}$};
\node at (4.0,0.39) {dim $\boldsymbol{m-r}$};
\node at (4.79,1.48) {\scalebox{1.1}{$\hbox{\textbf{R}}^{\text{\scriptsize{$\boldsymbol{m}$}}}$}};
\node at (4.6,2.8) {\scalebox{1.1}{$\hbox{\textbf{C}}(\boldsymbol{A})$}};
\node at (4.9,0.7) {\scalebox{1.1}{$\hbox{\textbf{N}}(\boldsymbol{A^{\pmb{+}}})=\hbox{\textbf{N}}(\boldsymbol{A^\mathrm{T}})$}};
\node at (4.45,2.45) {$\boldsymbol{b}_\text{\footnotesize ${\boldsymbol{c}}$}$};%
\end{tikzpicture}
}
\vspace{-5pt}
\caption{The four subspaces for $A^{\pmb{+}}$ are the four subspaces for $A^\mathrm{T}$.}\label{fig:2}
\end{figure}
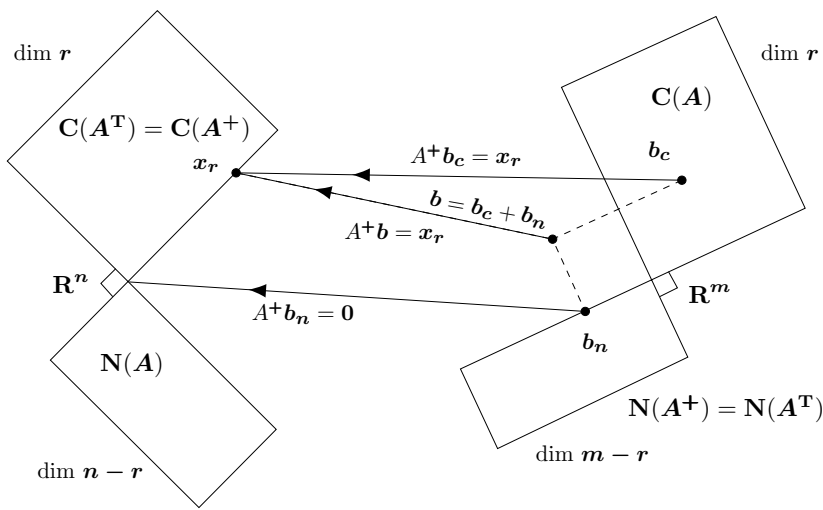

\newpage
\noindent
The proof of Theorem 1 begins with a simple Lemma.

\begin{lemma}\label{lemma:C-deflation-N-inflation}
We have $\boldsymbol{\mathrm{C}}(A)=\boldsymbol{\mathrm{C}}(AB)$ if and only if $\mathrm{rank}(A)=\mathrm{rank}(AB)$, and $\boldsymbol{\mathrm{N}}(B)=\boldsymbol{\mathrm{N}}(AB)$ if and only if $\mathrm{rank}(B)=\mathrm{rank}(AB)$.
\end{lemma}
\begin{proof}
Always $\boldsymbol{\mathrm{C}}(A)\supset\boldsymbol{\mathrm{C}}(AB)$ and $\boldsymbol{\mathrm{N}}(AB)\supset\boldsymbol{\mathrm{N}}(B)$. In each pair, equality of dimensions guarantees equality of spaces. The column space can only become smaller and the nullspace larger.
\end{proof}

\begin{proof}[Proof of Theorem~\ref{thm:rol-formula}]
First, we will show $\boldsymbol{\mathrm{N}}(A^\mathrm{T})\subset\boldsymbol{\mathrm{N}}(R^+C^+)$ and $\boldsymbol{\mathrm{N}}(A^\mathrm{T})\supset\boldsymbol{\mathrm{N}}(R^+C^+)$, and that $\boldsymbol{\mathrm{C}}(A^\mathrm{T})\subset\boldsymbol{\mathrm{C}}(R^+C^+)$ and $\boldsymbol{\mathrm{C}}(A^\mathrm{T})\supset\boldsymbol{\mathrm{C}}(R^+C^+)$ when $\mathrm{rank}(C)=\mathrm{rank}(R)=r$. Then, we will show that $R^+C^+$ perfectly inverts the action of $A$ on the same subspaces. 

\vspace{5pt}
\noindent
\textbf{Step 1\,:}\quad $\hbox{\textbf{N}}(A^\mathrm{T})\subset\hbox{\textbf{N}}(R^+C^+)$. Every vector $\boldsymbol{b_n}$ in the nullspace of $A^\mathrm{T}$ is also in the nullspace of $R^+C^+$.

\vspace{5pt}
\noindent
If $A^\mathrm{T}\boldsymbol{b_n}=0$, then $R^\mathrm{T}C^\mathrm{T}\boldsymbol{b_n}=0$. Since $R$ has full row rank and $C$ has full column rank, it follows that we must have $C^\mathrm{T}\boldsymbol{b_n} = 0$. Then we also have $(C^\mathrm{T}C)^{-1}C^\mathrm{T}\boldsymbol{b_n}=C^+\boldsymbol{b_n}=0$. Therefore, $\boldsymbol{b_n}$ is in the nullspace of $C^+$. Since $R$ is full row rank, matrix $R^+$ has full column rank, and we have $\hbox{\textbf{N}}(C^+)=\hbox{\textbf{N}}(R^+C^+)$. Every $\boldsymbol{b_n}$ in the nullspace of $A^\mathrm{T}$ is also in the nullspace of $R^+C^+$.

\vspace{5pt}
\noindent
\textbf{Step 2\,:}\quad $\hbox{\textbf{N}}(A^\mathrm{T})\supset\hbox{\textbf{N}}(R^+C^+)$. Every vector $\boldsymbol{b_n}$ in the nullspace of $R^+C^+$ is also in the nullspace of $A^\mathrm{T}$.

\vspace{5pt}
\noindent
Suppose that $R^+C^+\boldsymbol{b_n}=\boldsymbol{0}$. Since $RR^+=I_r$, we have $CRR^+C^+\boldsymbol{b_n} $ $= C(C^\mathrm{T}C)^{-1}C^\mathrm{T}\boldsymbol{b_n}= \boldsymbol{0}$. Since $C(C^\mathrm{T}C)^{-1}$ has full column rank, we see that $C^\mathrm{T}\boldsymbol{b_n}=\boldsymbol{0}$, so $\boldsymbol{b_n}$ is orthogonal to every column of $C$. Therefore $\boldsymbol{b_n}\in\hbox{\textbf{N}}(A^\mathrm{T})$.

\vspace{5pt}
\noindent
\textbf{Step 3\,:}\quad $\hbox{\textbf{C}}(A^\mathrm{T})\subset\hbox{\textbf{C}}(R^+C^+)$. Every vector $\boldsymbol{x}$ in the row space of $A$ is also in the column space of $R^+C^+$.

\vspace{5pt}
\noindent
We have $\boldsymbol{x} = A^\mathrm{T}\boldsymbol{y}=R^\mathrm{T}C^\mathrm{T}\boldsymbol{y}=R^\mathrm{T}[(RR^\mathrm{T})^{-1}RR^\mathrm{T}]C^\mathrm{T}\boldsymbol{y}=R^+R(R^\mathrm{T}C^\mathrm{T}\boldsymbol{y})=R^+R\boldsymbol{x}$. So by Lemma 1, we conclude that $\boldsymbol{x}\in\hbox{\textbf{C}}(R^+)=\hbox{\textbf{C}}(R^+C^+)$.

\vspace{5pt}
\noindent
\textbf{Step 4\,:}\quad $\hbox{\textbf{C}}(A^\mathrm{T})\supset\hbox{\textbf{C}}(R^+C^+)$. Every vector $\boldsymbol{x}$ in the column space of $R^+C^+$ is also in the row space of $A$.

\vspace{5pt}
\noindent
Since $C^T$ has full row rank, we have $\hbox{\textbf{C}}(R^TC^T) = \hbox{\textbf{C}}(R^T)$. Thus, $\hbox{\textbf{C}}(A^T) = \hbox{\textbf{C}}(R^T)$.
Similarly, since $C^+$ has full row rank, we have $\hbox{\textbf{C}}(R^+C^+) = \hbox{\textbf{C}}(R^+)$.
Then, by the property of the pseudoinverse, $\hbox{\textbf{C}}(R^+C^+) = \hbox{\textbf{C}}(R^+) = \hbox{\textbf{C}}(R^T) = \hbox{\textbf{C}}(A^T)$.

\vspace{5pt}
\noindent
\textbf{Step 5\,:}\quad The product of pseudoinverses $R^+C^+$ is equal to the pseudoinverse $A^+$.

\vspace{5pt}
\noindent
We have established that $R^+C^+$ has the same column space and nullspace as $A^+$, i.e., $\hbox{\textbf{C}}(R^+C^+) = \hbox{\textbf{C}}(A^T)$ and $\hbox{\textbf{N}}(R^+C^+) = \hbox{\textbf{N}}(A^T)$. Now, we inspect the projectors. We have $AR^+C^+ = CRR^+C^+  = CC^+$. This is the orthogonal projector onto $\hbox{\textbf{C}}(C)$. 
By the rank assumption and by Lemma \ref{lemma:C-deflation-N-inflation}, we also have
$\hbox{\textbf{C}}(A)=\hbox{\textbf{C}}(CR)=\hbox{\textbf{C}}(C)$.

Similarly, $R^+C^+A = R^+C^+CR = R^+R$, which is the orthogonal projector onto $\hbox{\textbf{C}}(R^T)$. It follows that $\hbox{\textbf{C}}(A^T)=\hbox{\textbf{C}}(R^TC^T)=\hbox{\textbf{C}}(R^T)$.

Since $AR^+C^+$ is the unique orthogonal projector onto $\hbox{\textbf{C}}(A)$ and $R^+C^+A$ is the unique orthogonal projector onto $\hbox{\textbf{C}}(A^T)$, we conclude $A^+=R^+C^+$. 

\end{proof}
\vspace{10pt}

Theorem~\ref{thm:rol-formula} gives only \textbf{sufficient} conditions for the reverse order law. That means we can have $(CR)^+ = R^+C^+$ despite $C$ not having full column rank and $R$ not having full row rank. To see that, let us consider
$$
C = \begin{bmatrix} 1 & 1 \\ 1 & 1 \end{bmatrix} \quad \text{and} \quad 
R = \begin{bmatrix} 1 & 0 \\ 1 & 0 \end{bmatrix}
$$
Clearly, $C$ does not have full column rank since its columns are linearly dependent. Similarly, $R$ does not have full row rank since its rows are linearly dependent. But in this case
$$
(CR)^+ = 
\frac{1}{4} 
\begin{bmatrix} 1 & 1 \\ 0 & 0 \end{bmatrix}
=
\frac{1}{2} \begin{bmatrix} 1 & 1 \\ 0 & 0 \end{bmatrix}
\frac{1}{4} \begin{bmatrix} 1 & 1 \\ 1 & 1 \end{bmatrix} =
R^+C^+.
$$

The \textbf{necessary and sufficient} conditions for $A^+ = (CR)^+ = R^+C^+$ are surprisingly complex. Greville formulated them in \cite{greville1966note}. The reverse order law for pseudoinverse, $(CR)^+=R^+C^+$, holds if and only if 
\begin{equation}\label{eq:greville conditions}
\boldsymbol{\mathrm{C}}(RR^TC^T) \subset\boldsymbol{\mathrm{C}}(C^T)
\quad\mathrm{and}\quad 
\boldsymbol{\mathrm{C}}(C^TCR) \subset \boldsymbol{\mathrm{C}}(R).
\end{equation} 
We provide an explanation and the proof outline in Figure~\ref{fig:rol-comm-diag}. 

Starting in the row space $\hbox{\textbf{C}}(A^T) = \hbox{\textbf{C}}(R^TC^T)$, we need to reach the column space $\hbox{\textbf{C}}(A)=\hbox{\textbf{C}}(CR)$. First, matrix $R$ maps $\boldsymbol{x}_r = R^TC^T\boldsymbol{y}$  to $R\boldsymbol{x}_r$ in the column space $\hbox{\textbf{C}}(R)$, which by (\ref{eq:greville conditions}) must also be in the row space $\hbox{\textbf{C}}(C^T)$. Then matrix $C$ can take $R\boldsymbol{x}_r$ to the column space $\hbox{\textbf{C}}(CR)$, so we have $CR\boldsymbol{x}_r = \boldsymbol{b}_c$. 

That map is next inverted by the pseudoinverse $A^+ = R^+C^+$. First, matrix $C^+$ maps $\boldsymbol{b}_c = CR\boldsymbol{x}_r$ into the row space $\hbox{\textbf{C}}(C^T) = \hbox{\textbf{C}}(C^+)$. But, by (\ref{eq:greville conditions}), that must also be in the column space $\hbox{\textbf{C}}(R)$. Indeed, in a special case when the left inverse is $C^+=(C^TC)^{-1}C^T$, we see that $(C^TC)^{-1}C^TCR\boldsymbol{x}_r = R\boldsymbol{x}_r\in\hbox{\textbf{C}}(R)$. Then $R^+$ can map $R\boldsymbol{x}_r$ back to the row space $\hbox{\textbf{C}}(R^TC^T)$.

In other words, the reverse order law demands that 
\begin{align}
C^+C(RA^T)=RA^T 
\quad\text{and}\quad
RR^+(C^TA)=C^TA.
\end{align}
It follows that $C$ and $R$ solve the two-sided projection equation (see \cite{karpowicz2021theory}):
\begin{align}
C^+C(RR^TC^TC)RR^+ = RR^TC^TC.
\end{align}
For the full-rank factorization of Theorem~\ref{thm:rol-formula} the equation holds trivially with $C^+C = I_r = RR^+$.

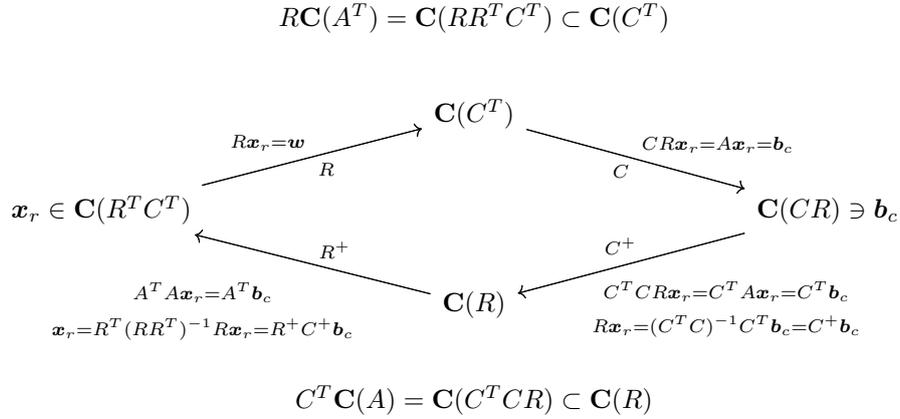
\begin{figure}[tb]
\centering
\[\begin{tikzcd}
& {R\boldsymbol{\mathrm{C}}(A^T) = \boldsymbol{\mathrm{C}}(RR^TC^T) \subset\boldsymbol{\mathrm{C}}(C^T)} \\
& {\hbox{\textbf{C}}(C^T)} \\
{\boldsymbol{x}_r\in\hbox{\textbf{C}}(R^TC^T)} && {\hbox{\textbf{C}}(CR)\ni\boldsymbol{b}_c} \\
& {\hbox{\textbf{C}}(R)} \\
& {C^T\boldsymbol{\mathrm{C}}(A)=\boldsymbol{\mathrm{C}}(C^TCR) \subset \boldsymbol{\mathrm{C}}(R)}
\arrow["{{R\boldsymbol{x}_r=\boldsymbol{w}}}"{pos=0.5}, from=3-1, to=2-2]
\arrow["{CR\boldsymbol{x}_r=A\boldsymbol{x}_r=\boldsymbol{b}_c}"{pos=0.5}, from=2-2, to=3-3]
\arrow["\substack{A^TA\boldsymbol{x}_r=A^T\boldsymbol{b}_c\\[.5em]\boldsymbol{x}_r=R^T(RR^T)^{-1}R\boldsymbol{x}_r=R^+C^+\boldsymbol{b}_c}"{pos=0.3}, from=4-2, to=3-1]
\arrow[""{name=0, anchor=center, inner sep=0}, "\substack{C^TCR\boldsymbol{x}_r=C^TA\boldsymbol{x}_r=C^T\boldsymbol{b}_c\\[1ex]R\boldsymbol{x}_r=(C^TC)^{-1}C^T\boldsymbol{b}_c=C^+\boldsymbol{b}_c}"{pos=0.7}, from=3-3, to=4-2]
\arrow["{R^+}"', from=4-2, to=3-1]
\arrow["R"', from=3-1, to=2-2]
\arrow["C"', from=2-2, to=3-3]
\arrow["{C^+}"'{pos=0.45}, from=3-3, to=4-2]
\end{tikzcd}\]
\caption{Given a full-rank decomposition $A=CR$, matrix $R$ maps $\hbox{\textbf{C}}(A^T)$ into $\hbox{\textbf{C}}(C^T)$ and matrix $C^T$ maps $\hbox{\textbf{C}}(A)$ into $\hbox{\textbf{C}}(R)$ if and only if $(CR)^+=R^+C^+$. }
\label{fig:rol-comm-diag}
\end{figure}

\section{The corrected formula}

It is not generally true that $A^+=R^+C^+$, so we should derive the correct formula presented in \cite{greville1966note}. As before, let us study the maps of subspaces in Figure~\ref{fig:corrected-formula-diag}. This section provides detailed illustration of how the correct formula must work. For the rigorous proof, see e.g. \cite{schott2016matrix}.

To reach the column space $\hbox{\textbf{C}}(CR)$ starting from the row space $\hbox{\textbf{C}}(R^TC^T)$, it must allow us to go through the intersection of $\hbox{\textbf{C}}(R)$ and $\hbox{\textbf{C}}(C^T)$. From there, $C$ maps into the column space $\hbox{\textbf{C}}(C)$. To go back, we need to invert the map $A = CR$. 

First, starting in $\hbox{\textbf{C}}(C)$ we must reach $\hbox{\textbf{C}}(C^T)$ intersecting $\hbox{\textbf{C}}(R)$. That is ensured by the orthogonal projection $RR^+$, which generates
$$
RR^+ \hbox{\textbf{C}}(C^T) = \hbox{\textbf{C}}(RR^+C^T)= \hbox{\textbf{C}}((CRR^+)^T) = \hbox{\textbf{C}}((CRR^+)^+)
$$
by the first principles. Then, to map $R\boldsymbol{x}_r$ in $\hbox{\textbf{C}}(C^T)$ back onto the row space $\hbox{\textbf{C}}(R^T)$, we need $(C^+CR)^+$ because
$$
R^T\hbox{\textbf{C}}(C^T) = \hbox{\textbf{C}}(R^T(C^+C)) = \hbox{\textbf{C}}((C^+CR)^T) = \hbox{\textbf{C}}((C^+CR)^+).
$$
As a result, $\hbox{\textbf{C}}(R^TC^T) = \hbox{\textbf{C}}(A^+)$ in Figure~\ref{fig:corrected-formula-diag}.

\begin{figure}[!ht]
\centering
\[\begin{tikzcd}
& {\hbox{\textbf{C}}((CRR^+)^+) = RR^+ \hbox{\textbf{C}}(C^T)} \\[10pt]
{
{\boldsymbol{x}_r\in\hbox{\textbf{C}}(R^TC^T)}} & {\hbox{\textbf{C}}(R)\cap\hbox{\textbf{C}}(C^T)} & {\hbox{\textbf{C}}(CR)\ni\boldsymbol{b}_c} \\
& {\hbox{\textbf{C}}(A^T)=\hbox{\textbf{C}}((C^+CR)^+)}
\arrow["{{R\boldsymbol{x}_r = (CRR^+)^+\boldsymbol{b}_c}}"', curve={height=18pt}, from=2-3, to=2-2]
\arrow["{{\boldsymbol{x}_r=(C^+CR)^+R\boldsymbol{x}_r}}"', curve={height=18pt}, from=2-2, to=2-1]
\end{tikzcd}\]
\caption{The pseudoinverse $A^+=(CR)^+=(C^+CR)^+(CRR^+)^+$ decomposes into the product of the pseudoinverse of $R$ projected on the row space of $C^T$ and the pseudoinverse of $C^T$ projected on the column space of $R$.}
\label{fig:corrected-formula-diag}
\end{figure}

Then, for any $\boldsymbol{b}_n$ in the nullspace of $R^TC^T$ we have $RR^+C^T\boldsymbol{b}_n = \boldsymbol{0}$. That shows that $\boldsymbol{b}_n$ is in the nullspace 
$$
\hbox{\textbf{N}}(RR^+C^T) = \hbox{\textbf{N}}((CRR^+)^+)\subset \hbox{\textbf{N}}((C^+CR)^+(CRR^+)^+) = \hbox{\textbf{N}}(A^+).
$$
We now have:
\begin{gather*}
C = \left[\begin{array}{rrr}1&0\\\end{array}\right] \quad 
R = \left[\begin{array}{rrr}1\\1\\\end{array}\right]\quad 
C^+C = \left[\begin{array}{rrr}1&0\\0&0\\\end{array}\right] \quad
RR^+ = \dfrac{1}{2}\left[\begin{array}{rrr}1&1\\1&1\\\end{array}\right]
\\[1em]
(C^+CR)^+(CRR^+)^+ = 
\left[\begin{array}{rrr}1&0\\\end{array}\right]
\left[\begin{array}{rrr}1\\1\\\end{array}\right]
= (CR)^+.
\end{gather*}
Therefore, we have reached the following theorem presenting the correct formula.

\vspace{10pt}
\begin{theorem}\label{thm:correct-formula}
The pseudoinverse $A^+$ of a product $A=CR$ is given by the product
\begin{equation}\label{eq:corrected}
A^+ = (CR)^+ = (C^+CR)^+(CRR^+)^+.
\end{equation}
\end{theorem}

\section{The generalized formula and interacting subspaces}

Factorization $A=CR$ expresses the classical elimination by row operations \cite{strang2022three,strang2023introduction}. It fills $C$ with the first $r$ independent columns of $A$. Those columns are a basis for the column space of $A$. Then column $j$ of $R$ specifies the combination of columns of $C$ that produces column $j$ of $A$. This example has column 3 = column 1 + column 2:
$$
A=\left[\begin{array}{rrr}1&4&5\\2&3&5\\\end{array}\right]=\left[\begin{array}{rr}1&4\\2&3\\\end{array}\right]\left[\begin{array}{rrr}1&0&1\\0&1&1\\\end{array}\right]=CR.
$$
That matrix $R$ contains the nonzero rows of the \textbf{reduced row echelon form} of $A$ and reveals the nullspace $\hbox{\textbf{N}}(A)=\hbox{\textbf{N}}(R)$. We will use that first-principles perspective to investigate generalized inverse maps and their fundamental subspaces. A good place to start is a general solution to linear equations. 

\subsection{General solution to $A^+\boldsymbol{b} = \boldsymbol{x}$}

As a full-rank decomposition, $A=CR$ leads to the following useful formula \cite{greville1960someapps,strang2023introduction}:
\begin{equation}
A^+=R^+C^+=R^\mathrm{T}(RR^\mathrm{T})^{-1}(C^\mathrm{T}C)^{-1}C^\mathrm{T}=R^\mathrm{T}(C^\mathrm{T}AR^\mathrm{T})^{-1}C^\mathrm{T}.
\label{eq:pinverse-formula}
\end{equation}
After a short inspection we can notice that the formula shows the ranks of $A^+$ and $A$ are equal. With equal ranks, we then have $\hbox{\textbf{C}}(A^+A)=\hbox{\textbf{C}}(A^+)$ by Lemma \ref{lemma:C-deflation-N-inflation}. See also Theorem 1.4.2 in \cite{ben2003generalized} or \cite{bjerhammar}.

Consider now the following equation 
$$
A^+\boldsymbol{b} = \boldsymbol{x}.
$$
For any given $\boldsymbol{x}$ and arbitrary $\boldsymbol{w}$, the general solution $\boldsymbol{b}$ to that equation is equal to
$$
\boldsymbol{b} = A\boldsymbol{x}+(I_m-AA^+)\boldsymbol{w}.
$$
That is easy to see, since $A^+$ satisfies the first and the second Penrose identity. For any $\boldsymbol{b}$ in $\hbox{\textbf{C}}(A)$ we recover the corresponding $\boldsymbol{x} = A^+\boldsymbol{b} = A^+A\boldsymbol{x} + (A-A^+AA^+)\boldsymbol{w} = A^+A\boldsymbol{x}$ in the row space $\hbox{\textbf{C}}(A^T)$. The map $A$ from the row space $\hbox{\textbf{C}}(A^T)$ to the column space $\hbox{\textbf{C}}(A)$ is inverted by the pseudoinverse $A^+$.  That is a wonderful and powerful property that holds when both the first and the second Penrose identity hold.

To see what happens when that is not the case, consider the following rank-deficient 
$$
A=
\left[\begin{array}{rrr}1&0\\0&0\\0&0\\\end{array}\right]
=
\left[\begin{array}{rrr}1\\0\\0\\\end{array}\right]
\left[\begin{array}{rrr}1&0\\\end{array}\right]
=C_0R_0.
$$
When we complete $C_0$ and $R_0$ to get square invertible matrices $[C_0, C_1]$ and $[R_0,  R_1]^T$, then the generalized factorization $A = CR$ yields
$$
A= 
\left[\begin{array}{rrr}C_0 & C_1\\\end{array}\right]
\left[\begin{array}{rrr}I_r & 0\\0& 0\\\end{array}\right]
\left[\begin{array}{rrr}R_0 \\R_1\\\end{array}\right].
$$
Given that decomposition of $A$, the first Penrose identity holds for any matrix defined by the following \textbf{generalized formula} \cite{marsaglia1974equalities}: 
\begin{align}\label{eq:geninv formula 1}
A^g = \left[\begin{array}{rrr}R_0 \\R_1\\\end{array}\right]^{-1}
\left[\begin{array}{rrr}I_r & Z_{12}\\Z_{21}& Z_{22}\\\end{array}\right]
\left[\begin{array}{rrr}C_0 & C_1\\\end{array}\right]^{-1}.
\end{align}
Matrix $A^g$ is a \textbf{generalized inverse} or \textbf{1-inverse} (since the first identity holds) of $A$. 

The second Penrose identity introduces additional constraints on how $A^g$ is constructed. We demand more, so additional constraints are what we should expect. Since
$$
A^gAA^g = \left[\begin{array}{rrr}R_0 \\R_1\\\end{array}\right]^{-1}
\left[\begin{array}{rrr}I_r & Z_{12}\\Z_{21}& Z_{21}Z_{12}\\\end{array}\right]
\left[\begin{array}{rrr}C_0 & C_1\\\end{array}\right]^{-1},
$$
the equation $A^gAA^g = A^g$ holds only when $Z_{22} = Z_{21}Z_{12}$. But then, matrix $A$ and its $1$-inverse $A^g$ have equal ranks. Therefore, any $1$-inverse $A^g$ of $A$ of the same rank is also $\{1,2\}$-inverse of $A$. The pseudoinverse $A^+$ of $A$ is a special example, as we have already seen.

\subsection{Generalized inverse fails to invert}

When $A$ and $A^g$ \textbf{do not }have equal ranks, the interactions of their fundamental subspaces become nontrivial. To better understand some of these interactions, we will now study the reduced row echelon forms of $A$ and $A^g$ and inspect how their nullspaces $\hbox{\textbf{N}}(A^T)$ and $\hbox{\textbf{N}}(A^g)$ interact. From there, we will give an illustration of the bigger picture.

Consider the following example:
$$
A^T=
\left[\begin{array}{rrr}1&0&0\\0&0&0\\\end{array}\right]
=\left[\begin{array}{rrr}1\\0\\\end{array}\right]
\left[\begin{array}{rrr}1&0&0\\\end{array}\right] = C_2R_2.
$$
Its nullspace $\hbox{\textbf{N}}(A^T) = \hbox{\textbf{N}}(R_2)$ is spanned by $(0,1,0)^T$ and $(0,0,1)^T$, and its dimension is $m-\mathrm{rank}(A^T) = 3-1 = 2$ (see Figure \ref{fig:1}). We can use formula (\ref{eq:geninv formula 1}) to design 
$$
A^g=\left[\begin{array}{rrr}1&3&2\\3&3&2\\\end{array}\right]
=
\left[\begin{array}{rrr}1&3\\3&3\\\end{array}\right]
\left[\begin{array}{rrr}1&0&0\\0&1&2/3\\\end{array}\right] = C_3R_3.
$$
The first Penrose identity holds, so $AA^gA = A$, but the second one does not. That matrix $A^g$ is only $1$-inverse of $A$ with the nullspace $\hbox{\textbf{N}}(A^g) = \hbox{\textbf{N}}(R_3)$ spanned by $(0,1,-3/2)^T$ of dimension $m-\mathrm{rank}(A^g) = 3-2 = 1$. 

The provided example demonstrates the case in which $\mathrm{rank}(A) =\mathrm{rank}(A^T) < \mathrm{rank}(A^g)$ and $AA^gA=A$. For this specific example, we observe the following subspace inclusions:
\begin{align*}
\hbox{\textbf{C}}(A^T)\subset\hbox{\textbf{C}}(A^g)
\quad&\text{and}\quad
\hbox{\textbf{C}}(A)\subset\hbox{\textbf{C}}((A^g)^T)
\\
\hbox{\textbf{N}}((A^g)^T)\subset\hbox{\textbf{N}}(A)
\quad&\text{and}\quad
\hbox{\textbf{N}}(A^g)\subset\hbox{\textbf{N}}(A^T).
\end{align*}

Indeed, we see that:
$$
\{(a,0)\colon a\in \mathbb{R}\} = \hbox{\textbf{C}}(A^T) 
\subset
\hbox{\textbf{C}}(A^g) = \mathbf{R}^2,
$$
so the row space of $A$ resides within the column space of $A^g$. Similarly, the column space of $A$ is a subspace of the row space of $A^g$, 
$$
\{(a,0,0)^T\colon a\in \mathbb{R}\} =\hbox{\textbf{C}}(A) 
\subset 
\hbox{\textbf{C}}((A^g)^T) = \mathrm{span}\{(1,3,2)^T,(3,3,2)^T\}.
$$
The nullspace relations are reversed,
$$
\{0\} = \hbox{\textbf{N}}((A^g)^T) \subset \hbox{\textbf{N}}(A) = \{(0,b)\colon b\in\mathbb{R}\}\subset\mathbb{R}^2
$$
and
$$
\{(0,b,-3/2b)^T\colon b\in \mathbb{R}\}=\hbox{\textbf{N}}(A^g) 
\subset 
\hbox{\textbf{N}}(A^T) = \{(0,b_2,b_3)^T\colon b_2,b_3\in\mathbb{R}\}.
$$

However, that is not the whole story, because there are also nonempty intersections of the subspaces. First, there is the case of $\hbox{\textbf{N}}(A)\cap\hbox{\textbf{C}}(A^g) = \{\boldsymbol{0}\}$.  Second, there is a much more interesting case of $\hbox{\textbf{N}}(A^T)\cap\hbox{\textbf{C}}((A^g)^T)$.

As before, for the sake of illustration, let us study an example presented in Figure \ref{fig:1-inv example}. 

Consider $\boldsymbol{y} = (0,3,2)^T$, which can be expressed as $(A^g)^T (\frac{3}{2},-\frac{1}{2})^T = \frac{3}{2} (1,3,2)^T - \frac{1}{2} (3,3,2)^T$ in the row space $\hbox{\textbf{C}}((A^g)^T)$. Since $A^T\boldsymbol{y} = \boldsymbol{0}$, this vector also belongs to the nullspace $\hbox{\textbf{N}}(A^T)$. 

When we apply $A^g$ to $\boldsymbol{y}$, we reach $\boldsymbol{x} = A^g\boldsymbol{y} = (13,13)^T$ in the column space $\hbox{\textbf{C}}(A^g)$. However, that vector is neither in the row space $\hbox{\textbf{C}}(A^T)$ nor in the nullspace $\hbox{\textbf{N}}(A)$. These two subspaces contain orthogonal projections of $\boldsymbol{x}$, denoted $\boldsymbol{x}_r$ and $\boldsymbol{x}_n$. 

Matrix $A$ cannot reconstruct the original $\boldsymbol{y}$ in $\hbox{\textbf{N}}(A^T)$ from $\boldsymbol{x} = A^g\boldsymbol{y}$. Instead, it produces $\bar{\boldsymbol{y}} = A\boldsymbol{x} = AG\boldsymbol{y} = (13,0,0)$ in its own column space $\hbox{\textbf{C}}(A)$, which also belongs to the row space $\hbox{\textbf{C}}((A^g)^T)$. For that $\bar{\boldsymbol{y}}$, we do find $\bar{\boldsymbol{x}} = A^gAA^g\boldsymbol{y} = (13,39)^T$ in the column space $\hbox{\textbf{C}}(A^g)$ that matrix $A^g$ can finally recover. Since $A^g$ is $1$-inverse of $A$, we have $\bar{\boldsymbol{y}} = AGAG\boldsymbol{y} = AG\boldsymbol{y} = (13,0,0)$, and $\bar{\boldsymbol{x}} = A^gAA^g\boldsymbol{y} = (13,39)^T$ again.

To get $\boldsymbol{x}$ from 
$$
\boldsymbol{b} = A\boldsymbol{x}+(I_m-AG)\boldsymbol{w},
$$
it is necessary to have $GA\boldsymbol{x}$ in the row space $\hbox{\textbf{C}}(A^T)$ and $(I_m-AG)\boldsymbol{w}$ in the nullspace $\hbox{\textbf{N}}(A^g)$. In the example above, neither condition is satisfied. Vector $\boldsymbol{x}$ is not in the row space of $A$ and the second Penrose identity does not hold, with vector $\boldsymbol{y}$ in $\hbox{\textbf{N}}(A^T)$ but not in $\hbox{\textbf{N}}(A^g)$. Also, $AG$ and $GA$ are not identity matrices. 

That interesting example brings us to the following conclusion. When $A^g$ is only $1$-inverse of $A$ and not its $\{1,2\}$-inverse, it fails to invert the mapping from the row space $\hbox{\textbf{C}}(A^T)$ to the column space $\hbox{\textbf{C}}(A)$. Some vectors in $\hbox{\textbf{N}}(A^T)$ intersecting with $\hbox{\textbf{C}}((A^g)^T)$ are mapped by $A^g$ to nonzero vectors in $\hbox{\textbf{C}}(A^g)$. Some vectors in $\hbox{\textbf{C}}((A^g)^T)$ are mapped by $A^g$ to vectors outside both $\hbox{\textbf{C}}(A^T)$ and $\hbox{\textbf{N}}(A)$, making them irrecoverable by $A$. Only vectors in $\hbox{\textbf{C}}(A)$ can be properly recovered using the 1-inverse $A^g$ of $A$.

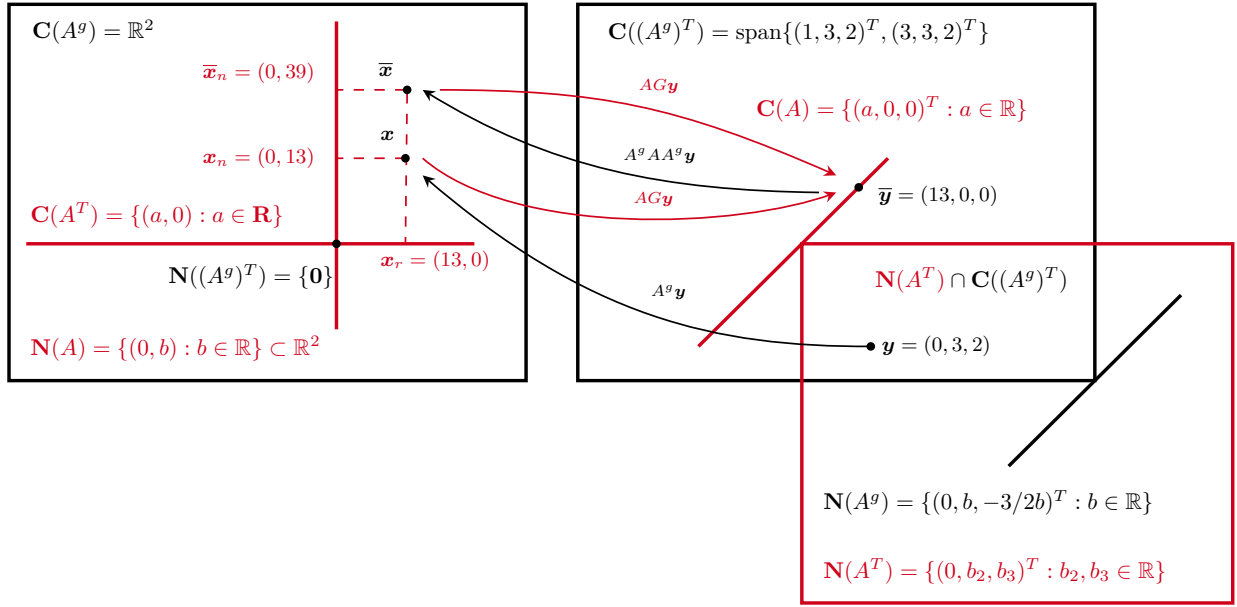
\begin{figure}[tb]
\centering    
\resizebox{\textwidth}{!}{

\tikzset{every picture/.style={line width=0.75pt}} 

\begin{tikzpicture}[x=0.75pt,y=0.75pt,yscale=-1,xscale=1]

\draw  [line width=1.5]  (30,40) -- (330,40) -- (330,260) -- (30,260) -- cycle ;
\draw  [line width=1.5]  (360,40) -- (660,40) -- (660,260) -- (360,260) -- cycle ;
\draw  [color={rgb, 255:red, 208; green, 2; blue, 27 }  ,draw opacity=1 ][line width=1.5]  (490,180) -- (740,180) -- (740,390) -- (490,390) -- cycle ;
\draw [color={rgb, 255:red, 208; green, 2; blue, 27 }  ,draw opacity=1 ][line width=1.5]    (540,130) -- (430,240) ;
\draw [color={rgb, 255:red, 208; green, 2; blue, 27 }  ,draw opacity=1 ][line width=1.5]    (300,180) -- (40,180) ;
\draw [color={rgb, 255:red, 208; green, 2; blue, 27 }  ,draw opacity=1 ][line width=1.5]    (220,230) -- (220,50) ;
\draw [color={rgb, 255:red, 0; green, 0; blue, 0 }  ,draw opacity=1 ][line width=1.5]    (710,210) -- (610,310) ;
\draw  [fill={rgb, 255:red, 0; green, 0; blue, 0 }  ,fill opacity=1 ] (218,180) .. controls (218,178.9) and (218.9,178) .. (220,178) .. controls (221.1,178) and (222,178.9) .. (222,180) .. controls (222,181.1) and (221.1,182) .. (220,182) .. controls (218.9,182) and (218,181.1) .. (218,180) -- cycle ;
\draw    (530,240) .. controls (439.45,240.25) and (365.74,223.42) .. (271.42,141.24) ;
\draw [shift={(270,140)}, rotate = 41.23] [fill={rgb, 255:red, 0; green, 0; blue, 0 }  ][line width=0.08]  [draw opacity=0] (7.14,-3.43) -- (0,0) -- (7.14,3.43) -- (4.74,0) -- cycle    ;
\draw [color={rgb, 255:red, 208; green, 2; blue, 27 }  ,draw opacity=1 ] [dash pattern={on 4.5pt off 4.5pt}]  (261,92) -- (260,180) ;
\draw [color={rgb, 255:red, 208; green, 2; blue, 27 }  ,draw opacity=1 ] [dash pattern={on 4.5pt off 4.5pt}]  (260,130) -- (240.5,130) -- (220,130) ;
\draw  [fill={rgb, 255:red, 0; green, 0; blue, 0 }  ,fill opacity=1 ] (258,130) .. controls (258,128.9) and (258.9,128) .. (260,128) .. controls (261.1,128) and (262,128.9) .. (262,130) .. controls (262,131.1) and (261.1,132) .. (260,132) .. controls (258.9,132) and (258,131.1) .. (258,130) -- cycle ;
\draw  [fill={rgb, 255:red, 0; green, 0; blue, 0 }  ,fill opacity=1 ] (521,147) .. controls (521,145.9) and (521.9,145) .. (523,145) .. controls (524.1,145) and (525,145.9) .. (525,147) .. controls (525,148.1) and (524.1,149) .. (523,149) .. controls (521.9,149) and (521,148.1) .. (521,147) -- cycle ;
\draw [color={rgb, 255:red, 208; green, 2; blue, 27 }  ,draw opacity=1 ]   (270,130) .. controls (323.19,176.54) and (453.51,170.93) .. (507.58,150.92) ;
\draw [shift={(510,150)}, rotate = 158.43] [fill={rgb, 255:red, 208; green, 2; blue, 27 }  ,fill opacity=1 ][line width=0.08]  [draw opacity=0] (7.14,-3.43) -- (0,0) -- (7.14,3.43) -- (4.74,0) -- cycle    ;
\draw [color={rgb, 255:red, 208; green, 2; blue, 27 }  ,draw opacity=1 ] [dash pattern={on 4.5pt off 4.5pt}]  (260,90) -- (240.5,90) -- (220,90) ;
\draw  [fill={rgb, 255:red, 0; green, 0; blue, 0 }  ,fill opacity=1 ] (259,90) .. controls (259,88.9) and (259.9,88) .. (261,88) .. controls (262.1,88) and (263,88.9) .. (263,90) .. controls (263,91.1) and (262.1,92) .. (261,92) .. controls (259.9,92) and (259,91.1) .. (259,90) -- cycle ;
\draw [color={rgb, 255:red, 0; green, 0; blue, 0 }  ,draw opacity=1 ]   (500,150) .. controls (410.9,148.76) and (344.34,136.5) .. (272.19,91.38) ;
\draw [shift={(270,90)}, rotate = 32.36] [fill={rgb, 255:red, 0; green, 0; blue, 0 }  ,fill opacity=1 ][line width=0.08]  [draw opacity=0] (7.14,-3.43) -- (0,0) -- (7.14,3.43) -- (4.74,0) -- cycle    ;
\draw [color={rgb, 255:red, 208; green, 2; blue, 27 }  ,draw opacity=1 ]   (507.22,138.77) .. controls (416.12,98.58) and (372.07,90.25) .. (280,90) ;
\draw [shift={(510,140)}, rotate = 203.92] [fill={rgb, 255:red, 208; green, 2; blue, 27 }  ,fill opacity=1 ][line width=0.08]  [draw opacity=0] (7.14,-3.43) -- (0,0) -- (7.14,3.43) -- (4.74,0) -- cycle    ;
\draw  [fill={rgb, 255:red, 0; green, 0; blue, 0 }  ,fill opacity=1 ] (528,240) .. controls (528,238.9) and (528.9,238) .. (530,238) .. controls (531.1,238) and (532,238.9) .. (532,240) .. controls (532,241.1) and (531.1,242) .. (530,242) .. controls (528.9,242) and (528,241.1) .. (528,240) -- cycle ;

\draw (41,154.4) node [anchor=north west][inner sep=0.75pt]  [font=\normalsize,color={rgb, 255:red, 208; green, 2; blue, 27 }  ,opacity=1 ]  {$\mathbf{C} (A^{T} )=\{(a,0):a\in \mathbf{R} \}$};
\draw (42,47.4) node [anchor=north west][inner sep=0.75pt]  [font=\normalsize]  {$\mathbf{C} (A^g)=\mathbb{R}^{2}$};
\draw (41,232.4) node [anchor=north west][inner sep=0.75pt]  [font=\normalsize,color={rgb, 255:red, 208; green, 2; blue, 27 }  ,opacity=1 ]  {$\mathbf{N} (A)=\{(0,b):b\in \mathbb{R} \}\subset \mathbb{R}^{2}$};
\draw (376,47.4) node [anchor=north west][inner sep=0.75pt]  [font=\normalsize]  {$\mathbf{C} ((A^g)^T )=\mathrm{span} \{(1,3,2)^{T} ,(3,3,2)^{T} \}$};
\draw (462,92.4) node [anchor=north west][inner sep=0.75pt]  [font=\normalsize,color={rgb, 255:red, 208; green, 2; blue, 27 }  ,opacity=1 ]  {$\mathbf{C} (A)=\{(a,0,0)^{T} :a\in \mathbb{R} \}$};
\draw (501,322.4) node [anchor=north west][inner sep=0.75pt]  [font=\normalsize]  {$\mathbf{N} (A^g)=\{(0,b,-3/2b)^{T} :b\in \mathbb{R} \}$};
\draw (501,362.4) node [anchor=north west][inner sep=0.75pt]  [font=\normalsize,color={rgb, 255:red, 208; green, 2; blue, 27 }  ,opacity=1 ]  {$\mathbf{N} (A^{T} )=\{(0,b_{2} ,b_{3} )^{T} :b_{2} ,b_{3} \in \mathbb{R} \}$};
\draw (531,192.4) node [anchor=north west][inner sep=0.75pt]  [font=\normalsize]  {$\mathbf{\textcolor[rgb]{0.82,0.01,0.11}{N}}\textcolor[rgb]{0.82,0.01,0.11}{(A}\textcolor[rgb]{0.82,0.01,0.11}{^{T}}\textcolor[rgb]{0.82,0.01,0.11}{)} \cap \mathbf{C} ((A^g)^T )$};
\draw (535,232.4) node [anchor=north west][inner sep=0.75pt]  [font=\small]  {$\boldsymbol{y} =(0,3,2)$};
\draw (121,190.4) node [anchor=north west][inner sep=0.75pt]    {$\mathbf{N} ((A^g)^T )=\{\mathbf{0} \}$};
\draw (244,112.4) node [anchor=north west][inner sep=0.75pt]  [font=\small]  {$\boldsymbol{x}$};
\draw (141,122.4) node [anchor=north west][inner sep=0.75pt]  [font=\small,color={rgb, 255:red, 208; green, 2; blue, 27 }  ,opacity=1 ]  {$\boldsymbol{x}_{n} =(0,13)$};
\draw (244,182.4) node [anchor=north west][inner sep=0.75pt]  [font=\small,color={rgb, 255:red, 208; green, 2; blue, 27 }  ,opacity=1 ]  {$\boldsymbol{x}_{r} =(13,0)$};
\draw (243,72.4) node [anchor=north west][inner sep=0.75pt]  [font=\small]  {$\overline{\boldsymbol{x}}$};
\draw (141,72.4) node [anchor=north west][inner sep=0.75pt]  [font=\small,color={rgb, 255:red, 208; green, 2; blue, 27 }  ,opacity=1 ]  {$\overline{\boldsymbol{x}}_{n} =(0,39)$};
\draw (401,202.4) node [anchor=north west][inner sep=0.75pt]  [font=\scriptsize]  {$A^g\boldsymbol{y}$};
\draw (391,148.4) node [anchor=north west][inner sep=0.75pt]  [font=\scriptsize]  {$\textcolor[rgb]{0.82,0.01,0.11}{AG}\boldsymbol{\textcolor[rgb]{0.82,0.01,0.11}{y}}$};
\draw (385,122.4) node [anchor=north west][inner sep=0.75pt]  [font=\scriptsize]  {$A^gAA^g\boldsymbol{y}$};
\draw (394,82.4) node [anchor=north west][inner sep=0.75pt]  [font=\scriptsize]  {$\textcolor[rgb]{0.82,0.01,0.11}{AG}\boldsymbol{\textcolor[rgb]{0.82,0.01,0.11}{y}}$};
\draw (533,144.4) node [anchor=north west][inner sep=0.75pt]  [font=\small]  {$\overline{\boldsymbol{y}} =(13,0,0)$};

\end{tikzpicture}

}
\caption{When $A^g$ is a generalized inverse for which only the first Penrose identity holds, or $1$-inverse for $A$, the fundamental subspaces interact in a nontrivial way. The provided example demonstrates the case in which $\mathrm{rank}(A^T) < \mathrm{rank}(A^g)$ and there are vectors in the left nullspace $\hbox{\textbf{N}}(A^T)$ that $A^g$ maps to nontrivial vectors in the column space  $\hbox{\textbf{C}}(A^g)$.}
\label{fig:1-inv example}
\end{figure}

\section{The general formula randomized }

It follows that the first Penrose identity only is enough to characterize the general solution of 
$$
A\boldsymbol{x} = \boldsymbol{b}.
$$
The second Penrose identity is not required. 

Given  
$$
\boldsymbol{x} = G\boldsymbol{b}+(I_n-GA)\boldsymbol{z},
$$
we have $A\boldsymbol{x} = AG\boldsymbol{b}+(A-AA^gA)\boldsymbol{z} = AG\boldsymbol{b} =\boldsymbol{b}$ for arbitrary $\boldsymbol{z}$ and $\boldsymbol{b} \in \colspace{A}$. That is summarized by the following famous result due to Penrose \cite{penrose1955generalized,ben2003generalized}:

\vspace{10pt}
\begin{theorem}\label{thm:Penrose}
Given $A$, $\bar{C}$ and $\bar{R}$, a necessary and sufficient condition for the equation $\bar{C}X\bar{R} = A$ to have a solution $X$ is that
\begin{align}
(\bar{C}G)A(H\bar{R}) = A,
\end{align}
where $G$ is a generalized 1-inverse of $\bar{C}$ and $H$ is a generalized 1-inverse of $\bar{R}$, i.e., $\bar{C}G\bar{C} = \bar{C}$ and $\bar{R}H\bar{R} = \bar{R}$. Then
\begin{align}
X = GAH + Z - (G\bar{C})Z(\bar{R}H),
\end{align}
where $Z$ is arbitrary.
\end{theorem}
\vspace{10pt}

\noindent
Given $A$, $\bar{C}$ and $\bar{R}$, we construct two projections, $\bar{C}G$ for the column space $\hbox{\textbf{C}}(A)$ and $H\bar{R}$ for the row space $\hbox{\textbf{C}}(A^T)$. They make the linear matrix equation work for $X$ and guarantee $\bar{C}X\bar{R}$ equals exactly $A$. That also  provides an elegant characterization of generalized 1-inverse $A^g$ (see e.g. \cite{ben2003generalized}). Consider
\begin{align}
X = A^g + Z - (A^gA)Z(AA^g),
\end{align}
where $Z$ is a random matrix. By the first Penrose identity, we immediately have:
\begin{align*}
AXA = AA^gA +AZA-AZA = AA^gA = A.
\end{align*}
Any $X$ defined by the randomized formula above is a 1-inverse of $A$. But that randomization is not the one we should be looking for.

\subsection{Pseudoinverse randomized }

In practical engineering contexts, great outcomes are achieved using low-rank randomized approximations such as Nystr{\"o}m or CUR decompositions and sketch-and-precondition techniques \cite{nakatsukasa2020fast,hamm2020perspectives,nakatsukasa2024sketch-and-precondition}. We believe this justifies the formal integration of randomization within the pseudoinverse formula to explore its randomized framework further. To do that, we can use Theorem \ref{thm:Penrose} again.

Let $P$ and $Q$ denote random sampling matrices. We can use them to construct the following $1$-inverses of $\bar{C}$ and $\bar{R}$,
\begin{align}
G_P = (P^T\bar{C})^+P^T
\quad\text{and}\quad
H_Q = Q(\bar{R}Q)^+.
\end{align}
It is easy to see from the properties of pseudoinverse, that $\bar{C}G_P\bar{C} = \bar{C}$ and $\bar{R}H_Q\bar{R} = \bar{R}$ under proper rank conditions. As a result we introduce oblique projectors, $\bar{C}G_P$ and $H_Q\bar{R}$. The first one projects onto the column space $\hbox{\textbf{C}}(A)$ along the nullspace $\hbox{\textbf{N}}(P^T)$. Similarly, the oblique projector $H_Q\bar{R}$ projects onto the row space $\hbox{\textbf{C}}(A^T)$ along the nullspace $\hbox{\textbf{N}}(Q)$. See illustration in Figure \ref{fig:oblique projection example},

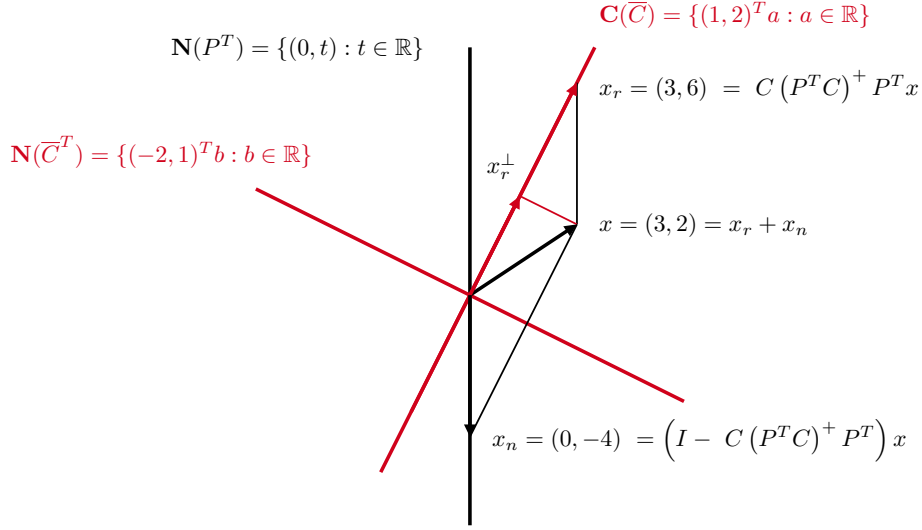
\begin{figure}[tb]
\centering    
\resizebox{.75\textwidth}{!}{
\tikzset{every picture/.style={line width=0.75pt}} 
\begin{tikzpicture}[x=0.75pt,y=0.75pt,yscale=-1,xscale=1]
\draw [color={rgb, 255:red, 208; green, 2; blue, 27 }  ,draw opacity=1 ][fill={rgb, 255:red, 208; green, 2; blue, 27 }  ,fill opacity=1 ][line width=1.5]    (450,60) -- (347,266) ;
\draw [color={rgb, 255:red, 0; green, 0; blue, 0 }  ,draw opacity=1 ][line width=0.75]    (440,80) -- (440,160) ;
\draw [line width=1.5]    (380,60) -- (380,330) ;
\draw [color={rgb, 255:red, 208; green, 2; blue, 27 }  ,draw opacity=1 ][fill={rgb, 255:red, 208; green, 2; blue, 27 }  ,fill opacity=1 ][line width=1.5]    (438.21,83.58) -- (330,300) ;
\draw [shift={(440,80)}, rotate = 116.57] [fill={rgb, 255:red, 208; green, 2; blue, 27 }  ,fill opacity=1 ][line width=0.08]  [draw opacity=0] (6.97,-3.35) -- (0,0) -- (6.97,3.35) -- cycle    ;
\draw [color={rgb, 255:red, 208; green, 2; blue, 27 }  ,draw opacity=1 ][fill={rgb, 255:red, 208; green, 2; blue, 27 }  ,fill opacity=1 ][line width=1.5]    (500,260) -- (260,140) ;
\draw [line width=1.5]    (380,200) -- (436.67,162.22) ;
\draw [shift={(440,160)}, rotate = 146.31] [fill={rgb, 255:red, 0; green, 0; blue, 0 }  ][line width=0.08]  [draw opacity=0] (6.97,-3.35) -- (0,0) -- (6.97,3.35) -- cycle    ;
\draw [color={rgb, 255:red, 208; green, 2; blue, 27 }  ,draw opacity=1 ]   (408,144) -- (440,160) ;
\draw [color={rgb, 255:red, 208; green, 2; blue, 27 }  ,draw opacity=1 ][line width=1.5]    (380,200) -- (406.21,147.58) ;
\draw [shift={(408,144)}, rotate = 116.57] [fill={rgb, 255:red, 208; green, 2; blue, 27 }  ,fill opacity=1 ][line width=0.08]  [draw opacity=0] (6.97,-3.35) -- (0,0) -- (6.97,3.35) -- cycle    ;
\draw [color={rgb, 255:red, 0; green, 0; blue, 0 }  ,draw opacity=1 ][line width=1.5]    (380,200) -- (380,276) ;
\draw [shift={(380,280)}, rotate = 270] [fill={rgb, 255:red, 0; green, 0; blue, 0 }  ,fill opacity=1 ][line width=0.08]  [draw opacity=0] (6.97,-3.35) -- (0,0) -- (6.97,3.35) -- cycle    ;
\draw [color={rgb, 255:red, 0; green, 0; blue, 0 }  ,draw opacity=1 ][fill={rgb, 255:red, 0; green, 0; blue, 0 }  ,fill opacity=0.37 ][line width=0.75]    (440,160) -- (380,280) ;
\draw (211,51.4) node [anchor=north west][inner sep=0.75pt]    {$\mathbf{N} (P^{T} )=\{(0,t):t\in \mathbb{R} \}$};
\draw (121,108.4) node [anchor=north west][inner sep=0.75pt]    {$\textcolor[rgb]{0.82,0.01,0.11}{\mathbf{N} (\overline{C}^{T} )=\{(-2,1)^{T} b:b\in \mathbb{R} \}}$};
\draw (451,32.4) node [anchor=north west][inner sep=0.75pt]    {$\textcolor[rgb]{0.82,0.01,0.11}{\mathbf{C} (\overline{C} )=\{(1,2)^{T} a:a\in \mathbb{R} \}}$};
\draw (451,152.4) node [anchor=north west][inner sep=0.75pt]    {$x=(3,2)=x_{r} +x_{n}$};
\draw (451,71.4) node [anchor=north west][inner sep=0.75pt]    {$x_{r} =(3,6)\ =\ C\left( P^{T} C\right)^{+} P^{T} x$};
\draw (391,268.4) node [anchor=north west][inner sep=0.75pt]    {$x_{n} =(0,-4)\ =\left( I-\ C\left( P^{T} C\right)^{+} P^{T}\right) x$};
\draw (388,118.4) node [anchor=north west][inner sep=0.75pt]    {$x_{r}^{\perp }$};
\end{tikzpicture}
}
\caption{For $\bar{C} = (1,2)^T$ and $P = (1,0)^T$ the oblique projector $\bar{C}G_P$ projects $\boldsymbol{x} = (3,2)^T$ onto the column space  $\hbox{\textbf{C}}(\bar{C})$ along the nullspace $\hbox{\textbf{N}}(P^T)$. That results in the decomposition $\boldsymbol{x}=\boldsymbol{x}_r+\boldsymbol{x}_n$. The orthogonal projection $\boldsymbol{x}_r^{\perp}$ is produced by the orthogonal projector $\bar{C}\bar{C}^+$ that acts along the nullspace $\hbox{\textbf{N}}(\bar{C}^T)$.}
\label{fig:oblique projection example}
\end{figure}

We now conclude our observations with the following new result in Theorem \ref{thm:general-MP-formula}.

\vspace{10pt}
\begin{theorem}\label{thm:general-MP-formula}
The pseudoinverse $A^+$ of a product $A=CR$ is equal to
\begin{equation}\label{eq:general-MP-formula}
A^+ = (P^TCR)^+P^TCRQ(CRQ)^+,
\end{equation}
if and only if $P$ and $Q$ satisfy the following rank-preservation condition:
\begin{equation}\label{eq:rank-preservation-MP-formula}
\mathrm{rank}(P^TA) = \mathrm{rank}(AQ) = \mathrm{rank}(A).
\end{equation}
\end{theorem}
\vspace{10pt}
\begin{proof}
If $A\boldsymbol{x}=0$, then $P^TA\boldsymbol{x}=0$, so $\hbox{\textbf{N}}(A) \subset \hbox{\textbf{N}}(P^TA)$. When $\mathrm{rank}(P^TA) = \mathrm{rank}(AQ) = \mathrm{rank}(A)$ for some $P$ and $Q$, then 
$$
\dim(\hbox{\textbf{N}}(A)) = n - \mathrm{rank}(A)\ \text{and}\ 
\dim(\hbox{\textbf{N}}(P^TA)) = n - \mathrm{rank}(P^TA).
$$
It follows that $\dim(\hbox{\textbf{N}}(A)) = \dim(\hbox{\textbf{N}}(P^TA))$ and, therefore, $\hbox{\textbf{N}}(A) = \hbox{\textbf{N}}(P^TA)$. That implies the orthogonal complements must also be equal, i.e., $\hbox{\textbf{C}}(A^T) = \hbox{\textbf{C}}((P^TA)^T).$

Since the subspaces of $A^T$ and $P^TA$ and $A^+$ are identical, the orthogonal projectors must be equal:
\begin{equation}\label{eq:proj_R}
(P^TA)^+(P^TA) = A^+A.
\end{equation}

By the same argument, we establish that
$$\hbox{\textbf{C}}(AQ) = \hbox{\textbf{C}}(A).$$
And, since $AA^+$ and $AQ(AQ)^+$ are the unique orthogonal projectors onto $\hbox{\textbf{C}}(A)$ and $\hbox{\textbf{C}}(AQ)$, with identical subspaces, we conclude that:
\begin{equation}\label{eq:proj_C}
AQ(AQ)^+ = AA^+.
\end{equation}
Then, if the rank preservation condition holds, we have:
\begin{align*}
(P^TA)^+P^TAQ(AQ)^+ 
&=
[(P^TA)^+P^TA]Q(AQ)^+ 
= A^+AQ(AQ)^+ 
\\
&= A^+[AQ(AQ)^+] = A^+AA^+ = A^+.
\end{align*}

Now we show that the rank preservation follows from the equality $A^+ = (P^TA)^+P^TAQ(AQ)^+$.  We have
\begin{align*}
\mathrm{rank}(A) &= \mathrm{rank}(A^+) \\
                 &= \mathrm{rank}\left((P^TA)^+P^TAQ(AQ)^+\right) \\
                 &\le \mathrm{rank}((P^TA)^+) \\
                 &= \mathrm{rank}(P^TA).
\end{align*}
Since it is always true that $\mathrm{rank}(P^TA) \le \mathrm{rank}(A)$, we must have $\mathrm{rank}(P^TA) = \mathrm{rank}(A)$. Similarly, since $\mathrm{rank}(AQ) \le \mathrm{rank}(A)$, we must have $\mathrm{rank}(AQ) = \mathrm{rank}(A)$. Thus, we have:
\begin{equation*}
\mathrm{rank}(P^TA) = \mathrm{rank}(AQ) = \mathrm{rank}(A),
\end{equation*}
as desired.
\end{proof}
\vspace{10pt}

The formula
\begin{equation}\label{eq:sketched-mp-formula}
  A^+ = (P^T C R)^+\, P^T C R Q\, (C R Q)^+
\end{equation}
admits a natural interpretation in terms of randomized sketching. It introduces two sketching matrices, an $m$ by $\ell_P$ matrix $P^T$ as a row sketch and $n$ by $\ell_Q$ matrix $Q$ as a column sketch. For a given $m$ by $n$ matrix $A$, the sketched matrices
$$
  P^T A = P^T C R 
  \quad \text{and}\quad 
  A Q = C R Q,
$$
are of sizes $\ell_P$ by $n$ and $m$ by $\ell_Q$, and can be substantially smaller than $m$ by $n$.

Theorem~\ref{thm:general-MP-formula} shows that
\eqref{eq:sketched-mp-formula} is exact if and only if the sketches preserve the rank of $A$. Thus, the sketched formula is an exact characterization of the Moore-Penrose pseudoinverse under the rank preservation conditions (\ref{eq:rank-preservation-MP-formula}).

\subsection{Two recipes for $P$ and $Q$}

One way to select rank-preserving matrices $P$ and $Q$ is to use a full-rank decomposition of $A$. A powerful example illustrating the idea comes with $A = \bar{Q}\bar{R}$ decomposition. By setting
\begin{align*}
P = \bar{Q} \ \mathrm{and}\ Q = \bar{R}^+
\end{align*}
we get 
$$A^+ = 
(\bar{Q}^T\bar{Q}\bar{R})^+\bar{Q}^T\bar{Q}\bar{R}\bar{R}^+(\bar{Q}\bar{R}\bar{R}^+)^+.
$$ 
However, since 
$$
\hbox{\textbf{C}}((\bar{Q}\bar{R}\bar{R}^+)^+) = \hbox{\textbf{C}}(\bar{R}\bar{R}^+\bar{Q}^T) = \hbox{\textbf{C}}(\bar{R}\bar{R}^+)\bar{Q}^T
\quad\mathrm{and}\quad
\hbox{\textbf{C}}((\bar{Q}^T\bar{Q}\bar{R})^+) = \bar{R}^T\hbox{\textbf{C}}(\bar{Q}^TQ),
$$
from Theorem \ref{thm:correct-formula} we obtain the following known formula:
\begin{align*}
A^+ = 
(\bar{Q}^T\bar{Q}\bar{R})^+\bar{Q}^T\bar{Q}\bar{R}\bar{R}^+(\bar{Q}\bar{R}\bar{R}^+)^+ 
=
(\bar{Q}^T\bar{Q}\bar{R})^+(\bar{Q}\bar{R}\bar{R}^+)^+ 
= \bar{R}^+\bar{Q}^T.
\end{align*}
Given
\begin{align*}  
A
=\left[\begin{array}{rrr}1&0&1\\0&2&0\\\end{array}\right] 
=
\left[\begin{array}{rrr}1&0\\0&1\\\end{array}\right] 
\left[\begin{array}{rrr}1&0&1\\0&2&0\\\end{array}\right] = 
\bar{Q}\bar{R},
\end{align*}
the pseudoinverse for $A$ is
\begin{align*}
A^+
=\left[\begin{array}{rrr}1/2&0\\0&1/2\\1/2&0\\\end{array}\right] 
=
\left[\begin{array}{rrr}1/2&0\\0&1/2\\1/2&0\\\end{array}\right] 
\left[\begin{array}{rrr}1&0\\0&1\\\end{array}\right] = 
\bar{R}^+\bar{Q}^T.
\end{align*}
\vspace{10pt}

Another way to get $P$ and $Q$ is to select random matrices. As an example, consider the pseudoinverse of
$$
A=\left[\begin{array}{rrr}1&4&5\\2&3&5\\\end{array}\right]=\left[\begin{array}{rr}1&4\\2&3\\\end{array}\right]\left[\begin{array}{rrr}1&0&1\\0&1&1\\\end{array}\right]=CR,
\ \text{which is equal to}
\
A^+=\dfrac{1}{15}\left[\begin{array}{rrr}-8&9\\7&-6\\-1&3\\\end{array}\right].
$$
Randomly selected matrices
$$
P=\left[\begin{array}{rrr}2&2&2\\1&2&2\\\end{array}\right]
\quad\text{and}\quad
Q=\left[\begin{array}{rrr}1&1\\0&2\\0&0\\\end{array}\right]
$$
preserve $\mathrm{rank}(A) = \mathrm{rank}(P^TA) = \mathrm{rank}(AQ) = 2$. Therefore, 
\begin{align*}
(P^TCR)^{+}P^TC 
=\dfrac{1}{3}\left[\begin{array}{rrr}2&-1\\-1&2\\1&1\\\end{array}\right]
\quad\text{and}\quad
RQ(CRQ)^+
=\dfrac{1}{5}\left[\begin{array}{rrr}-3&4\\2&-1\\\end{array}\right]
\end{align*}
reconstruct the pseudoinverse
\begin{align*}
A^+=
\dfrac{1}{3}\left[\begin{array}{rrr}2&-1\\-1&2\\1&1\\\end{array}\right]
\cdot
\dfrac{1}{5}\left[\begin{array}{rrr}-3&4\\2&-1\\\end{array}\right]
=
\dfrac{1}{15}\left[\begin{array}{rrr}-8&9\\7&-6\\-1&3\\\end{array}\right].
\end{align*}
\vspace{10pt}

\subsection{Generalized formula generalized}

Observe that the formula in Theorem \ref{thm:general-MP-formula} employs oblique projectors determined using pseudoinverses to guarantee that all four Penrose identities are satisfied for the reconstructed $A^+$. But we can as well use projectors that are established through generalized inverses to obtain a generalized $\{1,2\}$-inverse $A^g$ for $A$. Let us verify that $AA^gA = A$. 

\vspace{10pt}
\begin{theorem}\label{thm:general-1inv-formula}
Consider an $m$ by $p$ matrix $P$ and $n$ by $q$ matrix $Q$ satisfying the following condition:
\begin{equation}
\mathrm{rank}(P^TA) = \mathrm{rank}(AQ) = \mathrm{rank}(A).
\end{equation}
Then, for any $\{1\}$-inverse $(P^TA)^g$ and $\{1\}$-inverse $(AQ)^g$,
\begin{equation}\label{eq:general-1inv-formula}
A^g = (P^TA)^gP^TAQ(AQ)^g
\end{equation}
is a generalized $\{1,2\}$-inverse of $A$.
\end{theorem}
\vspace{10pt}
\begin{proof}
By assumption, $P^TA(P^TA)^gP^TA = P^TA$ and $AQ(AQ)^gAQ = AQ$. Therefore,
$$
P^TA(I-(P^TA)^gP^TA)=0
\quad \text{and}\quad
(I-AQ(AQ)^g)AQ = 0.
$$
Since $\mathrm{rank}(P^TA) = \mathrm{rank}(AQ) = \mathrm{rank}(A)$, that implies $\colspace{I-(P^TA)^gP^TA}\subset \nullspace{P^TA} = \nullspace{A}$. And, therefore, 
$$
A(P^TA)^gP^TA = A.
$$
since $A(I-(P^TA)^gP^TA) = A - A(P^TA)^gP^TA = 0.$
Similarly, we have $\colspace{AQ} = \colspace{A}$, which implies 
$$
AQ(AQ)^gA = A.
$$
We can now verify that the first and the second Penrose identity follows from the established properties of $\{1\}$-inverses.

The first identity is easy to see,
$$
A(P^TA)^g(P^TA)Q(AQ)^gA= AQ(AQ)^gA = A.
$$
The second identity is somewhat more hidden, but it is also true. For clarity of presentation, let $G_P = (P^TA)^g$ and $G_Q = (AQ)^g$. By associativity:
\begin{align*}
A^g A A^g &= G_P P^T (A Q G_Q A) G_P P^T A Q G_Q.
\end{align*}
Since $AQ G_Q A = A$:
\begin{align*}
A^g A A^g &= G_P P^T A G_P P^T A Q G_Q 
\\
&= G_P P^T (A G_P P^TA) Q G_Q.
\end{align*}
And since $A G_P P^TA = A$:
\begin{align*}
A^g A A^g &= G_P P^T A Q G_Q = A^g.
\end{align*}
Therefore, 
\begin{equation*}
A^g = (P^TCR)^gP^TCRQ(CRQ)^g
\end{equation*}
is a generalized $\{1,2\}$-inverse of $A = CR$.
\end{proof}
\vspace{10pt}

We can now transform the general formula in Theorem \ref{thm:general-1inv-formula} to the following compact formula in  Theorem \ref{thm:Nystroem-12ginv-formula} below (also, compare that with Theorem 1.6.5 in \cite{ben2003generalized}).

\vspace{10pt}
\begin{theorem}\label{thm:Nystroem-12ginv-formula}
Consider an $m$ by $p$ matrix $P$ and $n$ by $q$ matrix $Q$ satisfying the rank preservation condition:
\begin{equation}
\mathrm{rank}(P^TA) = \mathrm{rank}(AQ) = \mathrm{rank}(A).
\end{equation}
Let $P^TAQ = CR$ be a full rank factorization of the sketched matrix. Then 
\begin{equation}\label{eq:Nystroem-12ginv-formula}
A^g = Q(CR)^+P^T
\end{equation}
is a generalized \{1,2\}-inverse for $A$.
\end{theorem}
\vspace{10pt}
\begin{proof}
Let $A = C_A R_A$ be a full rank factorization of the $m$ by $n$ matrix $A$, where the $m$ by $r$ matrix $C_A$ has full column rank $r$ and the $r$ by $n$ matrix $R_A$ has full row rank $r$. Substituting this into the sketch $P^TAQ$, we obtain the factorization:
\begin{align*}
CR = (P^T C_A)(R_A Q) = KL.
\end{align*}
The rank preservation conditions imply that these factors retain full rank. Since $\mathrm{rank}(P^T A) = r$, we see that $\mathrm{rank}(P^T C_AR_A) = \mathrm{rank}(KR_A) = \mathrm{rank}(K) = r$. And, similarly, since $\mathrm{rank}(AQ) = r$, we have $\mathrm{rank}(C_AR_AQ) = \mathrm{rank}(C_AL) = \mathrm{rank}(L) = r$.

Since $K = P^TC_A$ has full column rank and $L = R_AQ$ has full row rank, the product $KL$ is a full rank factorization of the sketch $CR = P^TAQ$. Then, by Theorem \ref{thm:rol-formula} (the Reverse Order Law), the pseudoinverse is:
\begin{align*}
(CR)^+ = (KL)^+ = L^+ K^+.
\end{align*}
We can now verify the first two Penrose identities for $A^g = Q (CR)^+ P^T = Q L^+ K^+ P^T$.
The first identity yields:
\begin{align*}
A A^g A &= (C_A R_A) Q (L^+ K^+) P^T (C_A R_A) \\
        &= C_A (R_A Q) L^+ K^+ (P^T C_A) R_A \\
        &= C_A (L L^+) (K^+ K) R_A.
\end{align*}
Since $L$ has full row rank, $L L^+ = I_r$. Since $K$ has full column rank, $K^+ K = I_r$. Thus:
\begin{align*}
A A^g A = C_A I_r I_r R_A = C_A R_A = A.
\end{align*}
The second identity yields:
\begin{align*}
A^g A A^g &= Q (L^+ K^+) P^T (C_A R_A) Q (L^+ K^+) P^T \\
          &= Q L^+ K^+ (P^T C_A) (R_A Q) L^+ K^+ P^T \\
          &= Q L^+ (K^+ K) (L L^+) K^+ P^T \\
          &= Q L^+ I_r I_r K^+ P^T \\
          &= Q (L^+ K^+) P^T = A^g.
\end{align*}
That shows we have constructed the generalized $\{1,2\}$-inverse of $A$.
\end{proof}
\vspace{10pt}

Contrary to formula (\ref{eq:general-MP-formula}) in Theorem \ref{thm:general-MP-formula}, formula (\ref{eq:Nystroem-12ginv-formula}) does not satisfy the symmetry conditions without additional assumptions. That means $AA^g = C_A K^+ P^T$ is generally not symmetric. However, if $P$ and $Q$ represent a rotation of bases, symmetry is restored.

\vspace{10pt}
\begin{theorem}\label{thm:Nystroem-MP-formula}
If $P$ and $Q$ are square orthogonal matrices (i.e., $P^TP = I_m$ and $Q^TQ = I_n$), then 
\begin{equation}\label{eq:Nystroem-MP-formula}
A^+ = Q(P^TAQ)^+P^T.
\end{equation}
\end{theorem}
\vspace{10pt}
\begin{proof}
Since $P$ and $Q$ are square and orthogonal, $P^T = P^{-1}$ and $Q^T = Q^{-1}$. We can invert the sketching relation $CR = P^T A Q$ to recover $A$:
\begin{align*}
A = P (CR) Q^T.
\end{align*}
It now follows from the properties of pseudoinverse that
\begin{align*}
A^+ = (P (CR) Q^T)^+ = (Q^T)^T (CR)^+ P^T = Q (CR)^+ P^T.
\end{align*}
\end{proof}
\vspace{10pt}

\section{Examples and applications}

In this section we study how low-rank $A$ comes from low-rank $A^+$. First, we present basic implementations of formulas derived earlier. Then, we use them in several interesting examples.

By Theorem~\ref{thm:general-MP-formula}, rank-preserving matrices $P$ ($m$ by $p$) and $Q$ ($n$ by $q$) reconstruct the pseudoinverse $A^+$ ($n$ by $m$) of $A$ ($m$ by $n$) in formula (\ref{eq:general-MP-formula}). That formula describes the following simple (or naive) algorithm:
\vspace{10pt}
\begin{verbatim}
% Randomized pseudoinverse approximation of A
[m,n]  = size(A);  P = randn(m,p); Q = randn(n,q);
PTA    = P'*A;    AQ = A*Q; 
Arplus = pinv(PTA)*PTA*Q*pinv(AQ);
\end{verbatim}
\vspace{10pt}
An improved version comes directly from Theorem~\ref{thm:Nystroem-MP-formula}. One approach to implement it is to calculate the SVD of the easier-to-calculate sketched matrix $P^TAQ$ and then take its inverse, as the following code illustrates:
%
\begin{verbatim}
% Randomized pseudoinverse approximation of A
[m,n]   = size(A); P = orth(randn(m,p)); Q = orth(randn(n,q));
Arplus  = Q*pinv(P'*A*Q)*P';
\end{verbatim}
\vspace{10pt}
When only generalized inverse is required, we can use the algorithm given by Theorem \ref{thm:general-1inv-formula}.
\vspace{10pt}
\begin{verbatim}
% Randomized generalized 12-inverse approximation of A by Theorem 6
[m,n]   = size(A); P = randn(m,p); Q = randn(n,q);
[Qbar,Rbar] = qr(P'*A*Q,0);
Arplus  = (Q/Rbar)*(P*Qbar)';
\end{verbatim}
\vspace{10pt}
Introduction of $P$ and $Q$ allows us to see the general design patterns behind matrix decompositions of randomized linear algebra and of the pseudoinverse approximations they require, either directly or indirectly. To be more specific, let us consider the full-rank factorization:
\begin{align*}
AQ = CR,
\end{align*}
where $Q$ is random. When $R^+=R^T(RR^T)^{-1}$, then Theorem \ref{thm:general-MP-formula} reveals the following pattern:
\begin{align}\label{eq:PtCCplus}
A^+ \approx A^+_p =(P^TA)^+P^TCR(CR)^+ = (P^TA)^+P^TCC^+.
\end{align}
Good algorithms come from the good choice of $P$ and $Q$, and $C$ and $R$.

The \textbf{randomized SVD} algorithm (see Halko and Tropp \cite{halko2011finding}) makes an excellent choice of $C$ to encode the basis for the column space of $A$. First, it finds an orthogonal basis by calculating the QR decomposition:
\begin{align*}
AQ = \bar{Q}\bar{R},\quad\text{where}\quad \bar{Q}^T\bar{Q} = I.
\end{align*}
Then it sets $P = \bar{Q}$ to obtain
\begin{align*}
A^+ \approx A^+_p &= (P^TA)^+P^TAQ(AQ)^+ 
\\
&= (P^TA)^+P^T\bar{Q}\bar{R}(\bar{Q}\bar{R})^+
\\
&= (P^TA)^+P^T\bar{Q}\bar{Q}^T
\\
&= (\bar{Q}^TA)^+(\bar{Q}^T\bar{Q})\bar{Q}^T = (\bar{Q}^TA)^+\bar{Q}^T.
\end{align*}
The oblique projector $A(P^TA)^+P^T$ becomes an orthogonal one $A(\bar{Q}^TA)^+\bar{Q}^T$, projecting onto the column space $\hbox{\textbf{C}}(\bar{Q})$ along the nullspace $\hbox{\textbf{N}}(\bar{Q}^T)$, as in Figure \ref{fig:oblique projection example}. The following Matlab code illustrates the idea:
\vspace{10pt}
\begin{verbatim}
% Pseudoinverse approximation in randomized SVD 
[~,n]   = size(A);      Q = randn(n,q); 
[Qbar,~] = qr(A*Q,0);   P = Qbar;
Arplus  = pinv(P'*A)*P'; 
% Randomized SVD part calculates Ur, S, and V for a much smaller matrix P'*A:
% [Ur,S,V] = svd(P'*A,’econ’); Ar = P*Ur*S*V'
\end{verbatim}
\vspace{10pt}

Now, let us consider the \textbf{CUR factorization}. It approximates:
\begin{align*}
A \approx \bar{C}U\bar{R},
\end{align*}
where $C = A(:,\mathcal{J}) = AI_{n}(:,\mathcal{J}) = AQ$ is a subset of columns of $A$ indexed by $\mathcal{J}$, $R = A(\mathcal{I},:) = I_m^T(:,\mathcal{I})A = P^TA$ is a subset of rows of $A$ indexed by $\mathcal{I}$, and $U$ is the mixing matrix. As demonstrated in \cite{hamm2020perspectives}, the mixing matrix can be given the form of $U = A(\mathcal{I},\mathcal{J}) = I_m^T(\mathcal{I},:)AI_{n}(:,\mathcal{J})$. Then,
\begin{align*}
A^+ \approx \bar{R}^+U\bar{C}^+ = 
A(\mathcal{I},:)^+
A(\mathcal{I},\mathcal{J})
A(:,\mathcal{J})^+.
\end{align*}
That is formula in Theorems \ref{thm:general-MP-formula} and \ref{thm:Nystroem-MP-formula} with $P = I_m(:,\mathcal{I})$ and $Q = I_{n}(:,\mathcal{J})$, namely,
\begin{align*}
A^+\approx A^+_p 
&= (AI_m(:,\mathcal{I}))^+I_m(:,\mathcal{I})AI_n(:,\mathcal{J})(AI_n(:,\mathcal{J}))^+,
\end{align*}
as the following Matlab code clearly illustrates:
\vspace{10pt}
\begin{verbatim}
% Randomized pseudoinverse approximation in CUR decomposition
[m,n]   = size(A);
I       = randperm(m,k);    
J       = randperm(n,k);
Arplus  = pinv(A(I,:))*A(I,J)*pinv(A(:,J)); 
% or equivalently,
% In      = eye(n);         Im      = eye(m);
% Q       = In(:,J);        P       = Im(:,I);
% Arplus  = pinv(P'*A)*(P'*A*Q)*pinv(A*Q);
\end{verbatim}
\vspace{10pt}
Selecting rows and columns in $\mathcal{I}$ and $\mathcal{J}$ is crucial for ensuring approximation accuracy. Basic methods for this selection include pivoting, sampling, and randomized pivoting. For a comprehensive overview and analysis of accuracy in the CUR decomposition, refer to \cite{dong2023simpler,park2024accuracy}. Notably, these results are frequently derived using oblique projectors, as illustrated in \cite{karpowicz2021theory}.

\subsection{Sparse sensor placement}

We can apply matrix sampling strategies of Theorem \ref{thm:general-MP-formula} to address the {sparse sensor placement problem} \cite{manohar2018data,manohar2021optimal}. Then, we revisit the generalized Nystr{\"o}m decomposition, illustrating the role of the pseudoinverse approximation involved. 

The goal of {sparse sensor placement} is to construct a sparse measurement matrix $P_p$ that selects a subset of $p$ rows of a data (dictionary or basis) matrix $A$ that allows for an accurate reconstruction of observation $\boldsymbol{y}$. We want to calculate a good estimate $\boldsymbol{y}_e$ of $\boldsymbol{y}$ in the column space $\hbox{\textbf{C}}(A)$ based on selected samples $\boldsymbol{y_s} = P^T\boldsymbol{y}$ (we can only observe $\boldsymbol{y}$ at $p$ locations in $P^T$). That translates to the following low-rank projection:
\begin{align}
\boldsymbol{y}_e 
= A (P_p^TA)^+P_p^T \boldsymbol{y} 
= AA_p^+ \boldsymbol{y}. 
\end{align}
The column space projection involves an approximation $A^+ \approx (P_p^TA)^+P_p^T$. We want to find $P_p = I_m(:,\mathcal{I})$, where $\mathcal{I}$ is the subset of $p < \mathrm{rank}(A)$ rows. 

Well-known solutions come from (deterministic) pivoting in QR and LU decompositions. The permutation matrix $P$ of the full-rank column-pivoted QR decomposition is an excellent choice. The algorithm registers in $P$ the indices of columns of $A$ with the largest norm to improve numerical stability. To get rows, we calculate:
\begin{align*}
A^TP = \bar{Q}\bar{R}. 
\end{align*}
Then, applying (\ref{eq:general-MP-formula}) with $Q = \bar{Q}$ we get:
\begin{align*}
A^+ = (P^TA)^+P^T A Q(A Q)^+ 
&= 
(P^TA)^+P^T P \bar{R}^T\bar{Q}^T \bar{Q}(A \bar{Q})^+ 
\\
&=
(P^TA)^+ \bar{R}^T (A \bar{Q})^+ 
\\
&= 
(\bar{R}^T\bar{Q}^T)^+ \bar{R}^T (P\bar{R}^T\bar{Q}^T\bar{Q})^+
\\
&=
(P^TA)^+P^T.
\end{align*}
By selecting $p$ rows from $P$, we get the desired low-rank approximation:
\begin{align}
\boldsymbol{y}_e = A (P^T_pA)^+P^T_p \boldsymbol{y}.
\end{align}
When applied to $A$,
\begin{align}\label{eq:QRrecon}
A = AA^+A 
\approx
A 
(P^T_pA)^+P^T_p
A = AA^+_pA.
\end{align}
The result is the following Matlab code (see also \cite{brunton2019data}):
\vspace{10pt}
\begin{verbatim}
[~,~,smx] = qr(A','vector');
PT      = Im(smx(1:p), :); 
ys      = PT*y;     
ye      = A*pinv(PT*A)*ys;
Arplus  = pinv(PT*A)*PT;
Ar      = A*Arplus*A;
\end{verbatim}
\vspace{10pt}

Another solution with competitive properties is the LU decomposition with complete pivoting (row and column permutations):
\begin{align*}
BAD = LU.
\end{align*}
In this case, by Theorem \ref{thm:Nystroem-12ginv-formula} we have:
\begin{align}
A^g = D(BAD)^+B.
\end{align}
By selecting $p$ rows from $B$ and $q$ columns from $D$, we obtain the estimation:
\begin{align}\label{eq:LUSSP}
\boldsymbol{y}_e = AA^g_p \boldsymbol{y} = A
D_q(B_pAD_q)^+B_p\boldsymbol{y}.
\end{align}
Also, the following reconstruction equation holds:
\begin{align}\label{eq:LUrecon}
A = AA^gA 
\approx
A 
D_q(B_pAD_q)^+B_p
A.
\end{align}
The Matlab code illustrates the idea:
\vspace{10pt}
\begin{verbatim}
[L,W,B,D]   = lu(sparse(A)); 
Arplus      = D(:,1:q)*pinv(full(B(1:p,:)*A*D(:,1:q)))*B(1:p,:);
Ar          = A*Arplus*A;
\end{verbatim}
\vspace{10pt}

That brings us to the generalized Nystr{\"o}m decomposition. As demonstrated in \cite{nakatsukasa2020fast,karpowicz2021theory}, many randomized linear algebra algorithms can be written as two-sided projections that give rise to that decomposition. An $m$ by $n$ matrix $A$ can be approximated with two sketching matrices, an $n$ by $p$ matrix $\Omega_c$ for the column space, and $m$ by $p$ matrix $\Omega_r$ for the row space:
\begin{align}
\begin{aligned}
A \approx A\Omega_c(\Omega_r^*A\Omega_c)^{+}\Omega_r^*A = AA^+_pA.
\end{aligned}
\end{align}
The relationship with (\ref{eq:LUSSP}) and (\ref{eq:LUrecon}) is straightforward. By Theorem~\ref{thm:general-MP-formula}, the generalized Nystr{\"o}m decomposition requires the low-rank pseudoinverse approximation:
\begin{align}
\begin{aligned}
A^+ \approx A^+_p = \Omega_c(\Omega_r^*A\Omega_c)^{+}\Omega_r^* = Q(P^TAQ)^+P^T.
\end{aligned}
\end{align}
As a result, in \cite{nakatsukasa2020fast} we get the following implementation:
\begin{verbatim}
AQ      = A*Q; PTA = P'*A; 
[Qbar,Rbar] = qr(P'*A*Q,0);
Arplus  = (Q/Rbar)*(P*Qbar)';
Ar      = A*Arplus*A;
\end{verbatim}
In other words, low-rank $A$ comes from low-rank $A^+$.

\subsection{Effective resistance calculations}
\label{subsection:Effective resistance calculations}

Consider a network consisting of $n$ nodes connected by $m$ edges. Each row of the $m$ by $n$ incidence matrix $A$ representing that network shows which two nodes are connected. Edge $i$ leaves node $j$ when $A_{ij} = -1$, and goes to node $k$ from $i$ when $A_{ik} = 1$. Then, given a diagonal $m$ by $m$ matrix $W$ of edge conductance, the weighted Laplacian matrix describing the network is:
\begin{align}
L = A^T W A.
\end{align}
When external flows $\boldsymbol{f}$ are applied to the nodes, and $\boldsymbol{1}^T\boldsymbol{f} = 0$, the equilibrium potentials $\boldsymbol{x}$ satisfy the equation (see \cite{strang:cse} for more details):
\begin{align}
L\boldsymbol{x} = \boldsymbol{f}.
\end{align}
The Laplacian matrix is singular with rank $n-1$ for a connected graph, so the potentials are determined only up to an additive constant $a\boldsymbol{1}$. By Theorem~\ref{thm:Penrose}, the general equilibrium solution is
\begin{align}
\boldsymbol{x} = L^+\boldsymbol{f} + (I_n-L^+L)\boldsymbol{y} =  L^+\boldsymbol{f} + a\boldsymbol{1}.
\end{align}

The effective resistance $R_{ij}$ between nodes $i$ and $j$ measures how far apart (in terms of the overall energy or information flow) the nodes are in the network with edge conductance $W$. To determine that distance, we set $\boldsymbol{f} = \boldsymbol{e}_i - \boldsymbol{e}_j$ to find equilibrium potentials:
\begin{align}
\boldsymbol{x} = L^+ (\boldsymbol{e}_i - \boldsymbol{e}_j) + a\boldsymbol{1}.
\end{align}
Then, we calculate the potential difference $x_i - x_j = (\boldsymbol{e}_i-\boldsymbol{e}_j)^T\boldsymbol{x}$ to obtain the desired effective resistance:
\begin{align}
R_{ij} = x_i - x_j = (\boldsymbol{e}_i - \boldsymbol{e}_j)^T \boldsymbol{x} =
(\boldsymbol{e}_i - \boldsymbol{e}_j)^T L^+ (\boldsymbol{e}_i - \boldsymbol{e}_j).
\end{align}
The effective resistance matrix $R$ is given by the formula:
\begin{align}
R = \operatorname{diag}(L^+)\,\mathbf{1}^T + \mathbf{1}\,\operatorname{diag}(L^+)^T - 2L^+.
\end{align}

This important matrix is computed from the pseudoinverse of the graph Laplacian matrix. For large networks, that can be computationally prohibitive. The formulas derived in this paper offer another perspective for efficient calculations, complementing the large body of work on effective resistance \cite{Drfler2011KronRO,Dwaraknath2023TowardsOE,Young2013ANN2,Spielman2008GraphSB}.

One way is to calculate an approximation $L^+_p = Q(P^TLQ)^+P^T$ by random sampling with $P$ and $Q$. That immediately gives approximation:
\begin{align}
R_p = \operatorname{diag}(L_p^+)\,\mathbf{1}^T + \mathbf{1}\,\operatorname{diag}(L_p^+)^T - 2L_p^+.
\end{align}
However, we would also like to explore another way. 

Let us partition the nodes into disjoint subsets $\mathcal{S}$ and $\mathcal{T}$, where $\mathcal{S}\cup \mathcal{T} = \{1,...,n\}$. Consider $\mathcal{S} = \{i,j\}$ selecting only two nodes in the network. We would like to estimate the effective resistance between these two nodes. Taking $P = Q = I_n(:,\mathcal{S})$ and denoting $\boldsymbol{d} = [1, -1]^T$, we can consider the approximation:
\begin{align*}
\tilde{R}_{ij} = 
(\boldsymbol{e}_i - \boldsymbol{e}_j)^T L^+_p (\boldsymbol{e}_i - \boldsymbol{e}_j)
&=
(Q\boldsymbol{d})^TQ(P^TLQ)^+P^T
(P\boldsymbol{d})
\\
&=
\boldsymbol{d}^T 
\begin{bmatrix}
L_{ii} & L_{ij}\\L_{ji} & L_{jj}    
\end{bmatrix}^+ 
\boldsymbol{d} 
=
\boldsymbol{d}^T 
L_\mathcal{SS}^+ 
\boldsymbol{d}. 
\end{align*}

To examine how good an approximation that is in comparison to the exact value of effective resistance $R_{ij}$, it is good to start with the Kron-reduced form of $L$ or its Schur complement:
\begin{align*}
\hat{L}_{\mathcal{S}} = 
L_\mathcal{SS}-L_\mathcal{ST}L^+_\mathcal{TT}
L_\mathcal{TS} = L_\mathcal{SS} - E_{\mathcal{S}},
\end{align*}
given the permuted block decomposition for $\mathcal{S}$, namely:
\begin{align*}
L = 
\begin{bmatrix}
L_\mathcal{SS} & L_\mathcal{ST}
\\
L_\mathcal{TS} & L_\mathcal{TT}
\end{bmatrix}.
\end{align*}
As demonstrated in \cite{Drfler2011KronRO}, the effective resistance $R_{ij}$ is invariant with respect to that Kron reduction, i.e.,
\begin{align*}
R_{ij} 
= (\boldsymbol{e}_i - \boldsymbol{e}_j)^T L^+ (\boldsymbol{e}_i - \boldsymbol{e}_j)
= (\hat{\boldsymbol{e}}_i - \hat{\boldsymbol{e}}_j)^T \hat{L}^+_\mathcal{S} 
(\hat{\boldsymbol{e}}_i - \hat{\boldsymbol{e}}_j).
\end{align*}
That provides a convenient way for us to examine the relation between $R_{ij}$ and $\tilde{R}_{ij}$ for any $\mathcal{S} = \{i,j\}$. We now prove the following result.

\begin{theorem}\label{thm:effective resistance estimate}
Consider a connected graph described by a weighted positive semidefinite Laplacian $L$ and let $R_{ij}$ denote the effective resistance between nodes $i$ and $j$:
\begin{align}
R_{ij} = (\boldsymbol{e}_i - \boldsymbol{e}_j)^\top L^+ (\boldsymbol{e}_i - \boldsymbol{e}_j).
\end{align}
Let $\tilde{R}_{ij}$ be the approximation defined from the principal submatrix 
$L_{\mathcal{SS}}$ with $\mathcal{S} = \{i,j\}$ and $\boldsymbol{d} = [1,-1]^\top$ by
\begin{align}
\tilde{R}_{ij} = \boldsymbol{d}^\top L_{\mathcal{SS}}^+ \boldsymbol{d}.
\end{align}
Then, $\tilde{R}_{ij} \le R_{ij}$ for all $i,j$. 

Moreover, for any nodes $i,j,k$ the following inference rules hold,
\begin{align}\label{eq:inference-rules}
\begin{aligned}
\text{if}\ \tilde{R}_{ij} - \tilde{R}_{ik} < -\gamma,\ &\text{then}\  R_{ij} - R_{ik} < 0,\\
\text{if}\ \tilde{R}_{ij} - \tilde{R}_{ik} > \phantom{-}\gamma,\ &\text{then}\  R_{ij} - R_{ik} > 0,
\end{aligned}
\end{align}
with $\gamma = 2/\lambda_2(L)$.

Equivalently, the approximation error on the effective resistance differences satisfies the following inequality:
\begin{align}
\bigl| (\tilde{R}_{ij} - \tilde{R}_{ik}) - (R_{ij} - R_{ik}) \bigr|
\le \frac{2}{\lambda_2(L)} \qquad\text{for all distinct } i,j,k.
\end{align}
\end{theorem}

\begin{proof}

Let us first observe that for any symmetric $L\succeq 0$ such that $LL^+\boldsymbol{d} = \boldsymbol{d}$, the properties of pseudoinverse imply that $\boldsymbol{d}^T(L^+)^TL\boldsymbol{u} = \boldsymbol{u}^TLL^+\boldsymbol{d} = \boldsymbol{u}^T\boldsymbol{d}$. Thus,
\begin{align*}
2\boldsymbol{d}^T\boldsymbol{u}-\boldsymbol{u}^TL\boldsymbol{u} 
&= 
2\boldsymbol{d}^T\boldsymbol{u}-\boldsymbol{u}^TL\boldsymbol{u} 
- \boldsymbol{d}^TL^+\boldsymbol{d} + \boldsymbol{d}^TL^+\boldsymbol{d}
\\
&=
-(\boldsymbol{u}^T-\boldsymbol{d}^TL^+)L(\boldsymbol{u}-L^+\boldsymbol{d})
+ \boldsymbol{d}^TL^+\boldsymbol{d}
\\
&= 
-(\boldsymbol{u}-L^+\boldsymbol{d})^T L (\boldsymbol{u}-L^+\boldsymbol{d})
+ \boldsymbol{d}^TL^+\boldsymbol{d} 
\\
&\le  \boldsymbol{d}^TL^+\boldsymbol{d} 
= \sup\{2\boldsymbol{d}^T\boldsymbol{u}-\boldsymbol{u}^TL\boldsymbol{u}|\ 
\boldsymbol{u}\perp \nullspace{L}\}.
\end{align*}
Consider now a partition of $L$ according to $\mathcal{S}$ and its complement $\mathcal{T}$:
\begin{align*}
L = 
\begin{bmatrix}
L_\mathcal{SS} & L_\mathcal{ST}
\\
L_\mathcal{TS} & L_\mathcal{TT}
\end{bmatrix}.
\end{align*}
The Kron-reduced form of $L$ is the Schur complement:
\begin{align*}
\hat{L}_{\mathcal{S}} = 
L_\mathcal{SS}-L_\mathcal{ST}L^+_\mathcal{TT}
L_\mathcal{TS} = L_\mathcal{SS} - E_{\mathcal{S}}.
\end{align*}
Recall that, as demonstrated in \cite{Drfler2011KronRO}, the effective resistance $R_{ij}$ is invariant with respect to that Kron reduction.

It follows that $\hat{L}_{\mathcal{S}} \preceq L_{\mathcal{SS}}$, and so, for any $\boldsymbol{u}$ we have:
\begin{align*}
-\boldsymbol{u}^T\hat{L}_{\mathcal{S}}\boldsymbol{u} \ge
-\boldsymbol{u}^T{L}_{\mathcal{SS}}\boldsymbol{u}.
\end{align*}
We are interested exclusively in the resistance distance for $\boldsymbol{e}_i - \boldsymbol{e}_j = I_n(:,\{i,j\})\boldsymbol{d}$, where $\boldsymbol{d}\in\colspace{\hat{L}_{\mathcal{S}}}$. In that case, adding $2\boldsymbol{d}^T\boldsymbol{u}$ to both sides, we have:
\begin{align*}
2\boldsymbol{d}^T\boldsymbol{u}-\boldsymbol{u}^T\hat{L}_{\mathcal{S}}\boldsymbol{u} \ge
2\boldsymbol{d}^T\boldsymbol{u}-\boldsymbol{u}^T{L}_{\mathcal{SS}}\boldsymbol{u}.
\end{align*}
Taking the supremum over $\boldsymbol{u}$ preserves the inequality, therefore
\begin{align*}
\sup \{ 2\boldsymbol{d}^T \boldsymbol{u} - \boldsymbol{u}^T \hat{L}_{\mathcal{S}} \boldsymbol{u}|\ \boldsymbol{u}\perp \boldsymbol{1}_\mathcal{S}\} =
 \boldsymbol{d}^T \hat{L}_{\mathcal{S}}^+ \boldsymbol{d}
\ge 
\boldsymbol{d}^T L_{\mathcal{SS}}^+ \boldsymbol{d} = 
\sup\{ 2\boldsymbol{d}^T \boldsymbol{u} - \boldsymbol{u}^T L_{\mathcal{SS}} \boldsymbol{u}|\ \boldsymbol{u}\neq \boldsymbol{0} \},
\end{align*}
and so:
\begin{align*}
\boldsymbol{d}^T \hat{L}_{\mathcal{S}}^+ \boldsymbol{d} \ge
\boldsymbol{d}^T L_{\mathcal{SS}}^+ \boldsymbol{d}
\end{align*}
As a result, 
\begin{align*}
\tilde{R}_{ij}
= \boldsymbol{d}^T L_{\mathcal{SS}}^+ \boldsymbol{d}
\le \boldsymbol{d}^T \hat{L}_{\mathcal{S}}^+ \boldsymbol{d}
= R_{ij},
\end{align*}
and we conclude the approximation always underestimates the true resistance, $0 \le R_{ij} - \tilde{R}_{ij} = \epsilon_{ij}$.

Since $L$ is symmetric, there exists an orthonormal basis $V$ of eigenvectors corresponding to real-valued eigenvalues $0 = \lambda_1(L) \le \lambda_2(L) \le \dots \le \lambda_n(L)$. Therefore, given $L^+ = VS^+U^T$, we have
\begin{align*}
\|L^+\|_2\le \max\{0,1/\lambda_2(L),\dots,1/\lambda_n(L)\} = 1/\lambda_2(L).
\end{align*}
For any pair $p,q$ we have:
\begin{align*}
R_{pq}
= (\boldsymbol{e}_p - \boldsymbol{e}_q)^T L^+ (\boldsymbol{e}_p - \boldsymbol{e}_q)
\le \|\boldsymbol{e}_p - \boldsymbol{e}_q\|^2 \,\|L^+\|_2
= \frac{2}{\lambda_2(L)}.
\end{align*}
We have already established, that $0 \le R_{ij} - \tilde{R}_{ij} = \epsilon_{ij}$. From that we conclude: 
\begin{align*}
0 \le \epsilon_{ij} \le \frac{2}{\lambda_2(L)}.
\end{align*}
Since
\begin{align*}
|\epsilon_{ij} - \epsilon_{ik}| \le \max(\epsilon_{ij},\epsilon_{ik})
\le \frac{2}{\lambda_2(L)},
\end{align*}
we can see that:
\begin{align*}
\bigl| (\tilde{R}_{ij} - \tilde{R}_{ik}) - (R_{ij} - R_{ik}) \bigr|
\le \frac{2}{\lambda_2(L)}.
\end{align*}
This proves the stated additive error bound.

Finally, to establish the inference rules, set $\Delta = R_{ij} - R_{ik}$ and $\tilde{\Delta} = \tilde{R}_{ij} - \tilde{R}_{ik}$.
If $\tilde{\Delta} < -\gamma$ with $\gamma = 2/\lambda_2(L)$, then
\[
\Delta
= \tilde{\Delta} - (\tilde{\Delta} - \Delta)
\le \tilde{\Delta} + |\tilde{\Delta} - \Delta|
\le \tilde{\Delta} + \gamma
< -\gamma + \gamma = 0.
\]
Similarly, if $\tilde{\Delta} > \gamma$ then
\[
\Delta
= \tilde{\Delta} - (\tilde{\Delta} - \Delta)
\ge \tilde{\Delta} - |\tilde{\Delta} - \Delta|
> \gamma - \gamma = 0.
\]
This establishes the inference rules \eqref{eq:inference-rules}.

\end{proof}

Consider a graph with $n=4$ nodes, unit edge weights (conductances $W=I$), and the Laplacian matrix $L$:
\begin{align*}
L = \begin{pmatrix}
1 & -1 & 0 & 0 \\
-1 & 2 & -1 & 0 \\
0 & -1 & 2 & -1 \\
0 & 0 & -1 & 1
\end{pmatrix}.
\end{align*}
The eigenvalues of $L$ are $\{0, 2-\sqrt{2}, 2, 2+\sqrt{2}\}$, so the smallest non-zero eigenvalue is $\lambda_2(L) = 2-\sqrt{2}$. By Theorem \ref{thm:effective resistance estimate}, the upper bound on the error of the estimated resistance differences is:
\begin{align*}
\gamma = \frac{2}{\lambda_2(L)} = \frac{2}{2-\sqrt{2}} = 2+\sqrt{2} \approx 3.4142. 
\end{align*}
The true effective resistance $R_{ij}$ between nodes $i$ and $j$ can be calculated using the pseudoinverse:
\begin{align*}
L^+ = 
\dfrac{1}{8}\begin{pmatrix}
7 & 1 & -3 & -5
\\ 
1 & 3 & -1 & -3
\\ 
-3 & -1 & 3 & 1
\\
-5 & -3 & 1 & 7 
\end{pmatrix}.
\end{align*}
Let us verify the effective resistance between node $i = 1$ and nodes $j = 2, 3, 4$. From
\begin{align*}
R_{ij} = (\boldsymbol{e}_i - \boldsymbol{e}_j)^\top L^+ (\boldsymbol{e}_i - \boldsymbol{e}_j)
\end{align*}
we obtain $ R_{12}=1, R_{13}=2$ and $R_{14}=3$. 

Now, we can compare that with the estimate in Theorem~\ref{thm:effective resistance estimate} calculated for $\mathcal{S}=\{i,j\}$ and $\boldsymbol{d} = [1,-1]^\top$. We have
\begin{align*}
\tilde{R}_{ij} 
= \boldsymbol{d}^T L_{\mathcal{S}\mathcal{S}}^+ \boldsymbol{d} 
= [1,-1]
\begin{bmatrix}
L_{ii} & L_{ij}\\L_{ji} & L_{jj}
\end{bmatrix}^+
\begin{bmatrix}
1\\-1
\end{bmatrix}.
\end{align*}
The approximations are $\tilde{R}_{12} = 1$, $\tilde{R}_{13} = 1.5$ and $\tilde{R}_{14} = 2$.

Therefore, as expected, the property $R_{ij} \ge \tilde{R}_{ij}$ holds:
\begin{align*}
R_{12} = 1 \ge 1 = \tilde{R}_{12}, \quad R_{13} = 2 \ge 1.5 = \tilde{R}_{13}, \quad R_{14} = 3 \ge 2 = \tilde{R}_{14}.
\end{align*}
The absolute errors $\epsilon_{ij} = R_{ij}-\tilde{R}_{ij}$ are $\epsilon_{12} = 0$, $\epsilon_{13} = 0.5$ and $\epsilon_{14} = 1$, therefore
\begin{align*}
0 \le \epsilon_{ij} \le \gamma = 2/\lambda_2(L) \approx 3.4142.
\end{align*}
Finally, it is easy to see that whenever $\tilde{R}_{ij} - \tilde{R}_{ik} < -\gamma$, then $R_{ij} - R_{ik} < 0$ and the estimate preserves the ordering.

\section{Related Work}\label{sec:related-work}

The theory of generalized inverses has a rich history, with deep connections to modern randomized numerical linear algebra, as well as machine learning and modern AI models development. We survey the foundational results on pseudoinverses of matrix products and position our contributions within this landscape.

\subsection{Classical theory of generalized inverses}

The Moore-Penrose pseudoinverse was independently discovered by Moore using the concept of general reciprocal \cite{moore1920reciprocal}, by Bjerhammar in \cite{bjerhammar}, and famously characterized by Penrose through the four defining identities \cite{penrose1955generalized}. The comprehensive treatment by Ben-Israel and Greville \cite{ben2003generalized} remains the standard reference, consolidating decades of development in the field.

\subsubsection*{The reverse order law}
The question of when $(AB)^+ = B^+A^+$ holds was definitively answered by Greville in \cite{greville1966note}. He established that the reverse order law is valid if and only if
\begin{align}\label{eq:greville-conditions-rw}
\mathbf{C}(BB^*A^*) \subseteq \mathbf{C}(A^*)
\quad\text{and}\quad 
\mathbf{C}(A^*AB) \subseteq \mathbf{C}(B),
\end{align}
or equivalently, when $A^*ABB^* = BB^*A^*A$. Erdelyi \cite{erdelyi1967matrix} provided an independent characterization via partial isometries the same year. Theorem~\ref{thm:rol-formula} recovers a well-known sufficient condition: when $C$ has full column rank and $R$ has full row rank, the factors satisfy \eqref{eq:greville-conditions-rw} trivially with $C^+C = I_r = RR^+$.

\subsubsection*{Full-rank factorization formula}
The formula $A^+ = R^T(C^TAR^T)^{-1}C^T$ for a full-rank factorization $A = CR$ appears as Theorem~3.6.2 in \cite{ben2003generalized}. Historical notes  attribute the result to MacDuffee (1959, private communication) and subsequent refinements by Greville \cite{greville1960someapps}. Our presentation follows this classical approach but emphasizes the geometric perspective of the four fundamental subspaces.

\subsubsection*{Generalized inverses of products}
The extension to arbitrary products was developed by Cline \cite{cline1968inverses}, who studied partitioned matrices and representations for $(AB)^+$ without rank assumptions. Hartwig \cite{hartwig1986reverse} extended these results to triple products. The corrected formula in Theorem~\ref{thm:correct-formula}, namely $(CR)^+ = (C^+CR)^+(CRR^+)^+$, can be understood through this lens: the factors $(C^+CR)^+$ and $(CRR^+)^+$ handle the projections between intermediate subspaces. While the underlying mathematics is established, our presentation directly connects it to the four fundamental subspaces framework.

\subsubsection*{The $\{1,2\}$-inverse characterization}
Marsaglia and Styan \cite{marsaglia1974equalities} provided the complete characterization of $\{1\}$-inverses through formula \eqref{eq:geninv formula 1}. The constraint $Z_{22} = Z_{21}Z_{12}$ for obtaining $\{1,2\}$-inverses, which implies equal ranks of $A$ and $A^g$, is classical. Theorem~\ref{thm:general-1inv-formula} extends this by incorporating sketching matrices $P$ and $Q$ into the construction.

\subsection{Randomized numerical linear algebra}

The past two decades have witnessed remarkable developments in randomized algorithms for matrix computation. We focus on connections to pseudoinverse computation and low-rank approximation.

\subsubsection{Randomized SVD and matrix approximation}
The foundational work of Halko, Martinsson, and Tropp \cite{halko2011finding} established the randomized range-finder framework: given a target rank $k$ and oversampling parameter $p$, one computes $Y = A\Omega$ for a random $n \times (k+p)$ matrix $\Omega$, then orthogonalizes to obtain $Q$ with $\mathbf{C}(Q) \approx \mathbf{C}(A)$. The approximation $A \approx QQ^TA$ leads to randomized SVD via the smaller matrix $Q^TA$. While pseudoinverse computation is implicit---one obtains $A^+ \approx V\Sigma^+U^T$ from the randomized SVD---explicit formulas for $A^+$ in terms of sketching matrices are not developed in \cite{halko2011finding}.

The comprehensive survey by Martinsson and Tropp \cite{martinsson2020randomized} provides extensive treatment of randomized algorithms but similarly focuses on approximation rather than pseudoinverse reconstruction.

\subsubsection*{Generalized Nystr\"om approximation}
Nakatsukasa \cite{nakatsukasa2020fast} introduced the generalized Nystr\"om approximation for arbitrary $m \times n$ matrices:
\begin{align}\label{eq:gen-nystrom-rw}
\hat{A} = (A\Omega)(\Psi^TA\Omega)^+(\Psi^TA),
\end{align}
where $\Omega$ and $\Psi$ are sketching matrices. This formula involves the pseudoinverse of the core matrix $\Psi^TA\Omega$ and produces a low-rank approximation $\hat{A} \approx A$. For symmetric positive semidefinite matrices with $\Psi = \Omega$, one recovers the classical Nystr\"om approximation. Theorem~\ref{thm:Nystroem-MP-formula} shows that when $P$ and $Q$ are orthogonal and $P^TAQ = CR$ is a full-rank factorization, then $A^+ = Q(CR)^+P^T$ exactly---connecting the generalized Nystr\"om structure to pseudoinverse reconstruction rather than matrix approximation.

\subsubsection*{CUR decomposition}
The CUR decomposition approximates $A \approx CUR$ where $C$ contains selected columns, $R$ selected rows, and $U$ is a linking matrix. Mahoney and Drineas \cite{mahoney2009cur} established probabilistic guarantees, while Hamm and Huang \cite{hamm2020perspectives} provided the definitive treatment of exactness conditions. Their result states that $A = CU^+R$ and $A^+ = R^+UC^+$ hold if and only if $\mathrm{rank}(U) = \mathrm{rank}(A)$. 

The connection to our framework is immediate: setting $P = I_m(:,\mathcal{I})$ and $Q = I_n(:,\mathcal{J})$ as column-selection matrices in Theorem~\ref{thm:general-MP-formula} recovers the CUR pseudoinverse formula. Our contribution is the unified perspective showing that CUR, randomized SVD, and sparse sensor placement all emerge as special cases of the randomized formula \eqref{eq:general-MP-formula} with different choices of $P$ and $Q$.

\subsubsection*{Sketch-and-precondition methods}
Iterative approaches such as Blendenpik \cite{avron2010blendenpik} and LSRN \cite{meng2014lsrn} use sketching to construct preconditioners for iterative solvers rather than computing $A^+$ directly. Recent work by Epperly, Meier, and Nakatsukasa \cite{epperly2024fast} achieves stability through iterative refinement. These methods complement direct pseudoinverse formulas by providing practical algorithms for large-scale computation.

\subsubsection*{Effective resistance and graph Laplacians}

Effective resistance, defined as $R_{ij} = (\boldsymbol{e}_i - \boldsymbol{e}_j)^TL^+(\boldsymbol{e}_i - \boldsymbol{e}_j)$ for a graph Laplacian $L$, plays a fundamental role in spectral graph theory and network analysis \cite{spielman2012spectral,ghosh2008minimizing}. The classical bound $R_{ij} \leq 2/\lambda_2(L)$, relating effective resistance to the algebraic connectivity (Fiedler eigenvalue), is well established \cite{ghosh2008minimizing}.

D\"orfler and Bullo \cite{Drfler2011KronRO} proved that effective resistance is preserved exactly under Kron reduction (Schur complementation). Spielman and Srivastava \cite{spielman2008graph} provided multiplicative $(1\pm\varepsilon)$ approximation bounds for spectral sparsification. However, additive error bounds for effective resistance approximation in terms of algebraic connectivity appear to be new. Theorem~\ref{thm:effective resistance estimate} establishes that the principal submatrix approximation $\tilde{R}_{ij} = \boldsymbol{d}^TL_{\mathcal{SS}}^+\boldsymbol{d}$ always underestimates the true resistance with error bounded by $R_{ij} - \tilde{R}_{ij} \leq 2/\lambda_2(L)$, and that resistance orderings are preserved when differences exceed this threshold.

\subsection{Summary of contributions}

We think the novelty of this paper lies in the integration of classical generalized inverse theory with randomized linear algebra. Theorem~\ref{thm:general-MP-formula} provides a unifying lens through which established algorithms---randomized SVD, CUR decomposition, generalized Nystr\"om, and sparse sensor placement---can be understood as instances of randomized pseudoinverse computation with specific choices of sketching matrices $P$ and $Q$. The rank-preservation condition $\mathrm{rank}(P^TA) = \mathrm{rank}(AQ) = \mathrm{rank}(A)$ is the criterion separating exact reconstruction from low-rank approximation. From that unification new insights may emerge. One we present here concerns the effective resistance estimation with sketching in Theorem~\ref{thm:effective resistance estimate}. There is a beautiful link between the rank-preserving generalized inverse and the effective resistance.

\section{Conclusions}

We proposed a unified framework for the analysis of generalized inverses of matrix products, integrating classical generalized inverse theory with randomized numerical linear algebra. By revisiting the pseudoinverse of $A=CR$ from the perspective of the four fundamental subspaces, we have analyzed three central formulas, clarifying their geometric interpretations.

We believe that the principal contribution is the development of the generalized randomized framework presented in Theorems \ref{thm:general-1inv-formula} and \ref{thm:general-MP-formula}. We have rigorously demonstrated that the exact reconstruction of the Moore-Penrose inverse $A^+$ or a generalized $\{1,2\}$-inverse $A^g$ using sketched matrices $P^TA$ and $AQ$ depends on the preservation of the rank of $A$. This framework provides a lens through which established algorithms, such as randomized SVD and CUR decomposition, can be better understood as specific instances of randomized generalized inversion.

Furthermore, the application of this framework to the approximation of effective resistance in Theorem \ref{thm:effective resistance estimate} highlights its analytical utility. That approximation based on principal submatrices always underestimates the true resistance within a worst-case spectral bound while preserving the effective resistance ordering.

Exploration of randomized generalized inverses may still yield novel algorithms and deep theoretical insights into large-scale matrix computations, critical in development of modern AI models.


\begin{thebibliography}{10}
\expandafter\ifx\csname url\endcsname\relax
  \def\url#1{\texttt{#1}}\fi
\expandafter\ifx\csname urlprefix\endcsname\relax\def\urlprefix{URL }\fi
\expandafter\ifx\csname href\endcsname\relax
  \def\href#1#2{#2} \def\path#1{#1}\fi

\bibitem{ben2003generalized}
A.~Ben-Israel, T.~N. Greville, Generalized inverses: theory and applications,
  Vol.~15, Springer Science \& Business Media, 2003.

\bibitem{campbell2009generalized}
S.~L. Campbell, C.~D. Meyer, Generalized inverses of linear transformations,
  SIAM, 2009.

\bibitem{ilic2017algebraic}
D.~S.~C. Ili{\'c}, Y.~Wei, Algebraic properties of generalized inverses,
  Vol.~52, Springer, 2017.

\bibitem{halko2011finding}
N.~Halko, P.-G. Martinsson, J.~A. Tropp, Finding structure with randomness:
  Probabilistic algorithms for constructing approximate matrix decompositions,
  {SIAM Review} 53~(2) (2011) 217--288.

\bibitem{murray2023randomized}
R.~Murray, J.~Demmel, M.~W. Mahoney, N.~B. Erichson, M.~Melnichenko, O.~A.
  Malik, L.~Grigori, P.~Luszczek, M.~Dereziński, M.~E. Lopes, T.~Liang,
  H.~Luo, J.~Dongarra, Randomized numerical linear algebra : A perspective on
  the field with an eye to software (2023).
\newblock \href {http://arxiv.org/abs/2302.11474} {\path{arXiv:2302.11474}}.

\bibitem{martinsson2020randomized}
P.-G. Martinsson, J.~Tropp, Randomized numerical linear algebra: Foundations \&
  algorithms, arXiv:2002.01387 (2020).

\bibitem{penrose1955generalized}
R.~Penrose, {A generalized inverse for matrices}, in: {Mathematical Proceedings
  of the Cambridge Philosophical Society}, Vol.~51, Cambridge University Press,
  1955, pp. 406--413.

\bibitem{greville1966note}
T.~N.~E. Greville, Note on the generalized inverse of a matrix product, {Siam
  Review} 8~(4) (1966) 518--521.

\bibitem{karpowicz2021theory}
M.~P. Karpowicz, A theory of meta-factorization, arXiv preprint
  arXiv:2111.14385 (2021).

\bibitem{schott2016matrix}
J.~R. Schott, Matrix analysis for statistics, John Wiley \& Sons, 2016.

\bibitem{strang2022three}
G.~Strang, D.~Drucker, Three matrix factorizations from the steps of
  elimination, Analysis and Applications 20~(06) (2022) 1147--1157.

\bibitem{strang2023introduction}
G.~Strang, {Introduction to Linear Algebra, 6th edition}, {Wellesley-Cambridge
  Press}, 2023.

\bibitem{greville1960someapps}
T.~N.~E. Greville, \href{http://www.jstor.org/stable/2028054}{Some applications
  of the pseudoinverse of a matrix}, SIAM Review 2~(1) (1960) 15--22.
\newline\urlprefix\url{http://www.jstor.org/stable/2028054}

\bibitem{bjerhammar}
A.~Bjerhammar, A generalized matrix algebra, {Transactions of the Royal
  Institute of Technology} (1958).

\bibitem{marsaglia1974equalities}
G.~Marsaglia, G.~PH~Styan, Equalities and inequalities for ranks of matrices,
  Linear and multilinear Algebra 2~(3) (1974) 269--292.

\bibitem{nakatsukasa2020fast}
Y.~Nakatsukasa, Fast and stable randomized low-rank matrix approximation,
  arXiv:2009.11392 (2020).

\bibitem{hamm2020perspectives}
K.~Hamm, L.~Huang, Perspectives on cur decompositions, Applied and
  Computational Harmonic Analysis 48~(3) (2020) 1088--1099.

\bibitem{nakatsukasa2024sketch-and-precondition}
M.~Meier, Y.~Nakatsukasa, A.~Townsend, M.~Webb,
  \href{https://doi.org/10.1137/23M1551973}{Are sketch-and-precondition least
  squares solvers numerically stable?}, SIAM Journal on Matrix Analysis and
  Applications 45~(2) (2024) 905--929.
\newblock \href {http://arxiv.org/abs/https://doi.org/10.1137/23M1551973}
  {\path{arXiv:https://doi.org/10.1137/23M1551973}}, \href
  {https://doi.org/10.1137/23M1551973} {\path{doi:10.1137/23M1551973}}.
\newline\urlprefix\url{https://doi.org/10.1137/23M1551973}

\bibitem{dong2023simpler}
Y.~Dong, P.-G. Martinsson, {Simpler is better: a comparative study of
  randomized pivoting algorithms for CUR and interpolative decompositions},
  Advances in Computational Mathematics 49~(4) (2023) 66.

\bibitem{park2024accuracy}
T.~Park, Y.~Nakatsukasa, Accuracy and stability of cur decompositions with
  oversampling, arXiv preprint arXiv:2405.06375 (2024).

\bibitem{manohar2018data}
K.~Manohar, B.~W. Brunton, J.~N. Kutz, S.~L. Brunton, Data-driven sparse sensor
  placement for reconstruction: Demonstrating the benefits of exploiting known
  patterns, IEEE Control Systems Magazine 38~(3) (2018) 63--86.

\bibitem{manohar2021optimal}
K.~Manohar, J.~N. Kutz, S.~L. Brunton, Optimal sensor and actuator selection
  using balanced model reduction, IEEE Transactions on Automatic Control 67~(4)
  (2021) 2108--2115.

\bibitem{brunton2019data}
S.~L. Brunton, J.~N. Kutz, Data-Driven Science and Engineering: Machine
  Learning, Dynamical Systems, and Control, Cambridge University Press, 2019.

\bibitem{strang:cse}
G.~Strang,
  \href{https://epubs.siam.org/doi/abs/10.1137/1.9780961408817}{Computational
  Science and Engineering}, Wellesley-Cambridge Press, Philadelphia, PA, 2007.
\newblock \href
  {http://arxiv.org/abs/https://epubs.siam.org/doi/pdf/10.1137/1.9780961408817}
  {\path{arXiv:https://epubs.siam.org/doi/pdf/10.1137/1.9780961408817}}, \href
  {https://doi.org/10.1137/1.9780961408817}
  {\path{doi:10.1137/1.9780961408817}}.
\newline\urlprefix\url{https://epubs.siam.org/doi/abs/10.1137/1.9780961408817}

\bibitem{Drfler2011KronRO}
F.~D{\"o}rfler, F.~Bullo,
  \href{https://api.semanticscholar.org/CorpusID:10743012}{Kron reduction of
  graphs with applications to electrical networks}, IEEE Transactions on
  Circuits and Systems I: Regular Papers 60 (2011) 150--163.
\newline\urlprefix\url{https://api.semanticscholar.org/CorpusID:10743012}

\bibitem{Dwaraknath2023TowardsOE}
R.~V. Dwaraknath, I.~Karmarkar, A.~Sidford,
  \href{https://api.semanticscholar.org/CorpusID:259262023}{Towards optimal
  effective resistance estimation}, ArXiv abs/2306.14820 (2023).
\newline\urlprefix\url{https://api.semanticscholar.org/CorpusID:259262023}

\bibitem{Young2013ANN2}
G.~F. Young, L.~Scardovi, N.~E. Leonard,
  \href{https://api.semanticscholar.org/CorpusID:14473765}{{A New Notion of
  Effective Resistance for Directed Graphs—Part II: Computing Resistances}},
  IEEE Transactions on Automatic Control 61 (2013) 1737--1752.
\newline\urlprefix\url{https://api.semanticscholar.org/CorpusID:14473765}

\bibitem{Spielman2008GraphSB}
D.~A. Spielman, N.~Srivastava,
  \href{https://api.semanticscholar.org/CorpusID:8001711}{Graph sparsification
  by effective resistances}, Proceedings of the fortieth annual ACM symposium
  on Theory of computing (2008).
\newline\urlprefix\url{https://api.semanticscholar.org/CorpusID:8001711}

\bibitem{moore1920reciprocal}
E.~H. Moore, On the reciprocal of the general algebraic matrix, Bulletin of the
  American Mathematical Society 26 (1920) 294--295.

\bibitem{erdelyi1967matrix}
I.~Erd{\'e}lyi, On the matrix equation ax= $\lambda$bx, Journal of Mathematical
  Analysis and Applications 17~(1) (1967) 119--132.

\bibitem{cline1968inverses}
R.~E. Cline, Inverses of rank invariant powers of a matrix, SIAM Journal on
  Numerical Analysis 5~(1) (1968) 182--197.

\bibitem{hartwig1986reverse}
R.~E. Hartwig, The reverse order law revisited, Linear algebra and its
  applications 76 (1986) 241--246.

\bibitem{mahoney2009cur}
M.~W. Mahoney, P.~Drineas, Cur matrix decompositions for improved data
  analysis, Proceedings of the National Academy of Sciences 106~(3) (2009)
  697--702.

\bibitem{avron2010blendenpik}
H.~Avron, P.~Maymounkov, S.~Toledo, Blendenpik: Supercharging lapack's
  least-squares solver, SIAM Journal on Scientific Computing 32~(3) (2010)
  1217--1236.

\bibitem{meng2014lsrn}
X.~Meng, M.~A. Saunders, M.~W. Mahoney, Lsrn: A parallel iterative solver for
  strongly over-or underdetermined systems, SIAM Journal on Scientific
  Computing 36~(2) (2014) C95--C118.

\bibitem{epperly2024fast}
E.~N. Epperly, M.~Meier, Y.~Nakatsukasa, Fast randomized least-squares solvers
  can be just as accurate and stable as classical direct solvers,
  Communications on Pure and Applied Mathematics (2024).

\bibitem{spielman2012spectral}
D.~Spielman, Spectral graph theory, Combinatorial scientific computing 18~(18)
  (2012).

\bibitem{ghosh2008minimizing}
A.~Ghosh, S.~Boyd, A.~Saberi, Minimizing effective resistance of a graph, SIAM
  Review 50~(1) (2008) 37--66.

\bibitem{spielman2008graph}
D.~A. Spielman, N.~Srivastava, Graph sparsification by effective resistances,
  in: Proceedings of the fortieth annual ACM symposium on Theory of computing,
  2008, pp. 563--568.

\end{thebibliography}

\end{document}